\documentclass[11pt,letterpaper]{amsart}

\usepackage{pictexwd,dcpic}
\usepackage{latexsym}
\usepackage[all,cmtip]{xy}
\usepackage{times}
\usepackage{braket}
\usepackage[english]{babel}
\usepackage[utf8]{inputenc}
\usepackage{faktor}
\usepackage{paralist,url,verbatim}
\usepackage{amscd}
\usepackage[T1]{fontenc}
\usepackage{xspace}
\usepackage{comment}
\usepackage{amsmath,amsthm,amssymb,amsfonts}
\usepackage{mathtools}
\usepackage{enumitem}
\usepackage{wasysym}

\usepackage[dvipsnames,svgnames,x11names]{xcolor}
\definecolor{darkblue}{rgb}{0.0, 0.0, 0.55}

\usepackage[colorlinks,linkcolor=darkblue,citecolor=darkblue,urlcolor=OliveGreen,hypertexnames=darkblue,pagebackref=false]{hyperref}

\definecolor{darkgreen}{RGB}{0,80,0}

\newcommand{\cT}{\mathcal{T}}

\newcommand{\ktrash}[1]{}

\newcommand{\change}[1]{}
\definecolor{cKlaus}{rgb}{0.15,0.40,0.03}
\definecolor{cKLinkGBU}{rgb}{1,0,0}  % Good Bad Ugly - the NEW Macros
\definecolor{cKLink}{rgb}{0.6,0.2,0.3}
\definecolor{cALink}{rgb}{0,0.3,0}
\usepackage[textsize=tiny]{todonotes}
\usepackage{tikz}
\usetikzlibrary{decorations.markings,intersections,positioning,calc}

  \tikzset{mylabel/.style  args={at #1 #2  with #3}{
    postaction={decorate,
    decoration={
      markings,
      mark= at position #1
      with  \node [#2] {#3};
 } } } }

\theoremstyle{plain}
\newtheorem{theorem}{Theorem}[section]
\newtheorem{conjecture}[theorem]{Conjecture}

\newtheorem{corollary}[theorem]{Corollary}
\newtheorem{construction}[theorem]{Construction}
\newtheorem{lemma}[theorem]{Lemma}
\newtheorem{proposition}[theorem]{Proposition}

\theoremstyle{definition}

\newtheorem{remark}[theorem]{Remark}
\newtheorem{definition}[theorem]{Definition}

\newtheorem{example}[theorem]{Example}

\newcommand{\0}{$0$-mutable}

\newcommand{\Hom}{\operatorname{Hom}}

\newcommand{\NN}{\ensuremath{\mathbb{N}}}
\newcommand{\PP}{\ensuremath{\mathbb{P}}}
\newcommand{\QQ}{\ensuremath{\mathbb{Q}}}

\newcommand{\RR}{\ensuremath{\mathbb{R}}}
\newcommand{\CC}{\ensuremath{\mathbb{C}}}

\newcommand{\ZZ}{\ensuremath{\mathbb{Z}}}

\newcommand{\GL}{\operatorname{GL}}

\newcommand{\paar}{(X,\partial X)}
\newcommand{\bfa}{\ensuremath{\mathbf{a}}}
\newcommand{\bfb}{\ensuremath{\mathbf{b}}}
\newcommand{\bfj}{\ensuremath{\mathbf{j}}}

\newcommand{\bfk}{\ensuremath{\mathbf{k}}}
\newcommand{\bfn}{\ensuremath{\mathbf{n}}}

\newcommand{\bft}{\ensuremath{\mathbf{t}}}
\newcommand{\bfx}{\ensuremath{\mathbf{x}}}

\newcommand{\bfz}{\ensuremath{\mathbf{z}}}
\newcommand{\bfZ}{\ensuremath{\mathbf{Z}}}

\newcommand{\conv}{\operatorname{conv}}

\newcommand{\n}{n}

\newcommand{\kenditem}{\vspace{-1ex}\end{itemize}}

\newcommand{\kitem}{\begin{itemize}\vspace{-2ex}}

\newcommand{\tM}{\widetilde{M}}
\newcommand{\tN}{\widetilde{N}}

\newcommand{\Der}{\operatorname{Der}}
\newcommand{\bound}{{\partial}} % Projektion auf Rand
\renewcommand{\0}{$0$-mutable}
\newcommand{\spec}{\operatorname{Spec}}

\newcommand{\lan}{\langle}
\newcommand{\ran}{\rangle}

\newcommand{\newt}{\Delta}
\newcommand{\mut}{\operatorname{mut}}

\newcommand{\tr}{\tilde{r}}

\usepackage{tikz-cd}

\newcommand{\cE}{\mathcal{E}}

\newcommand{\cI}{\mathcal{I}}

\newcommand{\cM}{\mathcal{M}}

\newcommand{\cS}{\mathcal{S}}

\newcommand{\cX}{\mathcal{X}}
\newcommand{\cY}{\mathcal{Y}}

\newcommand{\bfT}{\ensuremath{\mathbf{t}}}
\newcommand{\bfTT}{\ensuremath{\mathbf{T}}}
\newcommand{\bfi}{\mathbf{i}}

\newcommand{\bfe}{\ensuremath{\mathbf{e}}}

\newcommand{\bfy}{\ensuremath{\mathbf{y}}}

\newcommand{\bo}{\partial}

        \renewcommand{\bo}{\bound}

      \newcommand{\s}{s}

\newcommand{\tw}{\tilde{w}}

\begin{document}

\title[Laurent polynomials and deformations of toric singularities]{Laurent polynomials and deformations of non-isolated Gorenstein toric singularities}
\author{Matej Filip }
\address{University of Ljubljana, Institute of Mathematics, Physics and Mechanics, Trzaska cesta 25, Ljubljana, Slovenia}
\email{matej.filip@fe.uni-lj.si}
\thanks{Supported by Slovenian Research Agency program P1-0222 and grant J1-60011}

\begin{abstract}
We establish a correspondence between one-parameter deformations of an affine Gorenstein toric pair $(X_P, \partial X_P)$, defined by a polytope $P$, and mutations of a Laurent polynomial $f$ with Newton polytope $\newt(f) = P$. For a Laurent polynomial $f$ in two variables, we construct a formal deformation of the three-dimensional Gorenstein toric pair $(X_{\newt(f)}, \partial X_{\newt(f)})$ over $\CC[[\bfTT_f]]$, where $\bfTT_f$ is the set of deformation parameters arising from mutations. The general fibre of this deformation is smooth if and only if $f$ is $0$-mutable. The Kodaira--Spencer map of the constructed deformation is injective, and if $f$ is maximally mutable, then the deformation cannot be nontrivially extended to a larger smooth base space. 
\end{abstract}



\maketitle

\section{Introduction}

The classification of Calabi--Yau and Fano manifolds is one of the most fundamental and extensively studied problems in geometry and theoretical physics, especially following the discovery of mirror symmetry in the late 1980s. In dimension two, the ten smooth Fano surfaces were classified by del Pezzo in the 1880s \cite{Pez87}. In dimension three, 105 types of smooth Fano threefolds were classified through the work of Fano in the 1930s and 1940s, Iskovskikh in the 1970s, and Mori–Mukai in the 1980s (see \cite{Fan47}, \cite{Isk77}, \cite{Isk78}, \cite{Isk79}, \cite{MM81} and \cite{MM03}). In higher dimensions, their classification remains an open problem (see \cite{KMM92}, \cite{BCHM10}, \cite{Bir19}, \cite{CCGGK13}, \cite{CCGK16}, \cite{KP22}, and \cite{CKPT21} for developments in higher dimensions).

Mirror symmetry originally describes the connection between two geometric objects called Calabi--Yau manifolds. If two Calabi--Yau manifolds are mirror symmetric, they are geometrically distinct yet equivalent when viewed from the physical side of string theory. Mirror symmetry has many mathematical formulations and generalisations that go beyond the Calabi--Yau manifolds. 
In addition to mirror-symmetric Calabi--Yau manifolds, the most notable conjecture involves the mirror relationship between Fano manifolds and Laurent polynomials (see \cite{CCGGK13}). 

Mirror symmetry provides new insight into classification, since it suggests that Fano manifolds are in  correspondence with certain Laurent polynomials. More precisely, if a Laurent polynomial $f$ is mirror to a Fano manifold $Y$, it is expected that a Fano manifold $Y$ admits a $\QQ$-Gorenstein degeneration to a singular toric variety, whose fan is the spanning fan of the Newton polytope $\newt(f)$.

There has been progress in constructing Fano manifolds using logarithmic geometry (see, e.g., \cite{FFR21}) and based on the above correspondence, the recent works \cite{CGR25}, \cite{CR24}, and \cite{Gr25} suggest that log structures arising from special Laurent polynomials are expected to produce new smooth Fano manifolds. There are also approaches that do not rely on logarithmic geometry (see, e.g., \cite{CKP19}) and in the recent work \cite{CHP24}, the smoothing of Gorenstein toric Fano 3-folds is constructed using admissible Minkowski decomposition data for a 3-dimensional reflexive polytopes.

%%%%%%
Deformation theory plays a fundamental role in the study of moduli spaces, families of varieties, and the local behavior of algebraic and geometric structures (see \cite{Ste03}, \cite{Ser06} and \cite{Har10}). In particular, the deformation theory of toric varieties has been extensively studied due to its connections with combinatorics and mirror symmetry.
 For affine toric surfaces, which are cyclic quotient singularities, Koll\'ar and Shepherd-Barron \cite{KS88} established a correspondence between certain partial resolutions (P-resolutions) and reduced miniversal base components. Additionally, Arndt \cite{Arn02} provided explicit equations for the miniversal base space. Furthermore, in \cite{Chr91} and \cite{Ste91}, Christophersen and Stevens produced a simpler set of equations for each reduced component of the miniversal base space. 

In higher dimensions, Altmann \cite{Alt97} constructed a miniversal deformation for affine Gorenstein toric varieties with isolated singularities. He further demonstrated that the reduced irreducible components can be explicitly described: they are in one-to-one correspondence with maximal decompositions of the defining polytope into Minkowski sums.

We aim to extend these results to non-isolated Gorenstein toric singularities. We will work with deformations of a pair $\paar$ instead of only $X$, since, in fact, we will see that it is more natural to work with the deformations of a pair, as we already noticed in \cite{CFP22} and \cite{Fil23}. 

In \cite{CFP22}, we formulated, together with Corti and Petracci, the following conjecture: there exists a canonical bijective correspondence  
$
\kappa\colon\mathfrak{B}\to \mathfrak{A},
$
where $\mathfrak{A}$ is the set of smoothing components of the three-dimensional affine toric Gorenstein pair $\paar$, and $\mathfrak{B}$ is the set of \0 Laurent polynomials with Newton polygon $P$, which defines $X$.
A Laurent polynomial is \0 if it can be mutated to a point (see Definition \ref{def 0mut} for a precise definition).

To state the main results of this paper precisely and clarify the techniques used in the proofs, we first introduce some definitions.
Consider a rank-$n$ lattice $\tN\cong \ZZ^n$ and its dual lattice $\tM=\mathrm{Hom}_{\ZZ} (\tN,\ZZ)$. A strictly convex
full-dimensional rational polyhedral cone $\sigma \subset \tN_\RR$ defines an affine toric variety
$X := \spec(A)$, where $A=\CC[\sigma^\vee \cap \tM]$. 
Note that $X$ is Gorenstein if and only if the primitive generators of the rays of $\sigma$ all lie on an affine hyperplane $(R^*=1)$ for some $R^*\in \tM$. This element $R^*$ is called the \emph{Gorenstein degree}. 
The polytope $P$ is defined as the convex hull of the primitive generators of the rays of $\sigma$, i.e.,  
$P:=\sigma \cap (R^* = 1).$
The isomorphism class of the toric variety $X$ depends only on the affine equivalence class of $P$. We denote by $X_P:=\spec \CC[S_P]:=\spec \CC[\sigma^\vee \cap \tM]$ the Gorenstein toric variety associated with $P$.

If $X_P$ has an isolated singularity, then the entire tangent space $T^1_{X_P}$ of the deformation functor of $X_P$ is concentrated in degree $-R^*$, i.e., $T^1_X(-R^*)=T^1_X$. 
This observation was crucial in Altmann's construction in \cite{Alt97}. Together with Altmann and Constantinescu, we provide partial results concerning the deformations of toric varieties at special lattice degrees in \cite{ACF22a} and \cite{ACF22b}, generalizing Altmann's results in \cite{Alt97}. 

A fundamental challenge in toric deformation theory is how to systematically combine deformations arising from different lattice degrees. We need to tackle this problem in order to understand deformations of the affine Gorenstein toric variety $X=X_P$ with non-isolated singularities. In this case, the polytope $P$ is arbitrary, and the tangent space $T^1_{(X,\partial X)}$ of the deformation functor of $\paar$ is spread over infinitely many lattice degrees in $\tM$:
$$
T^1_{\paar} = \bigoplus_{m \in \tM} T^1_{\paar}(-m).
$$
The same holds for $T^1_X$, which is naturally embedded in $T^1_{\paar}$.

We propose that \emph{Laurent polynomials} serve as a natural tool to connect deformations from different lattice degrees. Laurent polynomials and their mutations have so far been used for constructing certain one-parameter deformations (see \cite{Ilt12}).  
In this paper, we show how they can be used to construct multi-parameter deformations from different lattice degrees. More precisely, from Ilten’s construction we get a flat family over $\PP^1$ whose fibre over $0$ is $X_{\newt(f)}$ and whose fibre over $\infty$ is $X_{\newt(g)}$, where the Laurent polynomials $f$ and $g$ are connected by a mutation. The main geometric idea of this paper is to show that not only the toric varieties $X_{\newt(f)}$ and $X_{\newt(g)}$ are related by such a family over $\PP^1$, but in fact their entire deformation families are connected, with tangent spaces determined by all possible mutations of $f$ and $g$. In the following, we make this statement more precise, present the main results of the paper, and outline the techniques used in the proofs.

It is known that a \emph{deformation pair} $(m, Q)$ of a polytope $P$, consisting of $m \in \tM$ and a lattice polytope $Q \subset N_\RR$ (see Definition \ref{def def pair} for the precise definition), induces a Minkowski summand $Q$ of the polyhedron $\sigma \cap (m = 1)$. Consequently, by a result of Altmann (see \cite{Alt00}), such a pair gives rise to a one-parameter deformation of $X_P$. We denote the corresponding deformation parameter by $t_{(m,Q)}$. The first key result of this paper is the explicit construction of this one-parameter deformation in terms of deformations of the defining equations: $X_P=\spec \CC[\bfx,u]/(f_{\bfk}(\bfx,u)\mid \bfk\in \NN^r)$ (see \eqref{eq fbfk xu2}), and in Section \ref{sec 3} we construct $F_{\bfk}(\bfx,u,t_{(m,Q)})$ such that $F_{\bfk}(\bfx,u,0)=f_{\bfk}(\bfx,u)$, and such that the linear relations among the $f_\bfk(\bfx,u)$ lift to linear relations among the $F_{\bfk}(\bfx,u,t_{(m,Q)})$. By the equational criteria of flatness, this implies that we have a one-parameter formal deformation of $X_P$ (and also of a pair $(X_P,\partial X_P)$) over $\CC[[t_{(m,Q)}]]$, see Corollary \ref{cor 3.3}.

If $f\in \CC[N]$ is a Laurent polynomial, we define the notion of mutation (see Definition \ref{def mmut} for the definition of $f$ being $(m,g)$-mutable, where $m \in \tM$ and $g$ is a Laurent polynomial) and show that 
if $f$ is $(m,g)$-mutable, then $(m,\newt(g))$ forms a deformation pair of $\newt(f)$ (see Lemma \ref{lem mgex}). 
In the following, we assume that $N\cong \ZZ^2$, so that $f$ is a Laurent polynomial in two variables and $X_{\newt(f)}$ is a three-dimensional affine Gorenstein toric variety.
We say that $f$ is \emph{$m$-mutable} if $m \in \widetilde{M}$ is not constant on $\sigma\cap (R^*=1)$, and if $f$ is $(m, g)$-mutable with $\newt(g) \subset (\pi_M(m)=0)$ being a line segment of lattice length 1. If $f$ is $m$-mutable, then our construction in Section \ref{sec 3} also provides an explicit one-parameter deformation of $\paar$ over $\CC[[t_m]]$, where $X=X_{\newt(f)}$. %Our goal is to show that those one-parameter deformation can be joined connect deformations from different lattice degrees.

We introduce the following set of deformation parameters:  
$$
\bft_{f} := \{t_m ~|~ m\in \cM(f)\},
\text{ where }\cM(f):=\{m\in \tM\mid \text{$f$ is $m$-mutable}\}.$$ 
We say that a Laurent polynomial $f\in \CC[N]\cong \CC[\ZZ^2]$ is \emph{maximally mutable} if there does not exist Laurent polynomial $g$ such that $\bft_f \subsetneqq \bft_g$ (see Definition \ref{def maxmut}).  
 We say that a polygon $P$ is \emph{$m$-mutable}, for $m\in \tM$, if there exists a Laurent polynomial $f$ with $\newt(f)=P$ such that $f$ is $m$-mutable. We denote 
 $$\cT:=\{m\in \tM\mid \text{$P$ is $m$-mutable, and $\lan v,m \ran>0$ for some $v\in \sigma\cap (R^*=1)$}\},$$
 $$\bfTT_f:=\{t_m\mid \text{$m\in \cT\cap \cM(f)$}\}\subset \bft_f.$$
The following theorem is the main theorem of this paper.

\begin{theorem}\label{thmm}
For any Laurent polynomial $f\in \CC[N]\cong \CC[\ZZ^2]$, there exists a formal deformation of the affine Gorenstein toric pair $\paar$ over $\CC[[\bfTT_f]]$, 
where $X = X_{\newt(f)}$. The general fibre of this deformation is smooth if and only if $f$ is a $0$-mutable Laurent polynomial. Moreover, the associated Kodaira--Spencer map is injective, and if $f$ is maximally mutable, then this deformation cannot be non-trivially extended over $\CC[[\bfTT_f,t_m]]$, for any $m\in \cT\setminus \cM(f)$. 
\end{theorem}

We now outline the main ideas of the the proof of Theorem \ref{thmm}. 
In Section \ref{sec 4}, we give a connection between mutations of a Laurent polynomial $f$ in two variables and its mutation $\mut_m^gf$. We define the map $\psi_{(m,g)}:\tM\to \tM$ and prove that $f$ is $r$-mutable if and only if $\mut^g_m f$ is $\psi_{(m,g)}(r)$-mutable in Proposition \ref{prop mutability}. Moreover, we have a piecewise linear map $\xi_{(m,g)}: \tM\to \tM$ that bijectively maps $S_{\newt(f)}$ onto $S_{\newt(\mut_m^gf)}$ (see Lemma \ref{st lem}).
The maps $\psi_{(m,g)}$ and $\xi_{(m,g)}$ map $s \in \widetilde{M}$ to $s + z_1 m$ and $s + z_2 m$, respectively, for some $z_1, z_2 \in \mathbb{Z}$.

The first main step to prove Theorem \ref{thmm} is to show that  
$(X_{\newt(f)},\partial X_{\newt(f)})$ is unobstructed in $\bft_f$ if and only if  $(X_{\newt(\mut_m^gf)},\partial X_{\newt(\mut_m^gf)})$ is unobstructed in $\bft_{\mut_m^gf}$ (see Definition \ref{def unobs} for the definition of unobstructedness, which in particular implies the existence of a formal deformation of $(X_{\newt(f)},\partial X_{\newt(f)})$ over $\CC[[\bft_f]]$). This is shown in Theorem \ref{main th 1 un} and the key idea is as follows. Assume there is a formal deformation $$\{F_\bfk(\bfx,u,\bft_f)\mid \bfk\in \NN^r\}$$ of $(X_{\newt(f)},\partial X_{\newt(f)})$ over $\CC[[\bft_f]]$ (see Definition \ref{def 5757}), such that the restriction of $F_\bfk(\bfx,u,\bft_f)$ to $t_m$, i.e., to $F_\bfk(\bfx,u,0,\dots,0,t_m,0,\dots,0)$ coincides with the above mentioned one-parameter deformation over $\CC[[t_m]]$, for every $t_m\in \bft_f$ and $\bfk\in \NN^{r}$. Then, 
$$\{\mut_{t_{m}}F_\bfk(\bfx,u,\bft_f)\mid \bfk\in \NN^r\}$$
is a deformation of $(X_{\newt(\mut_m^gf)},\partial X_{\newt(\mut_m^gf)})$ over $\CC[[\bft_{\mut_m^gf}]]$, where $$\mut_{t_{m}}F_\bfk(\bfx,u,\bft_f)$$ is defined in Definition \ref{mmutable def}. Roughly speaking, $\mut_{t_{m}} F_\bfk(\bfx, u, \bft_f)$ is obtained from $F_\bfk(\bfx, u, \bft_f)$ by replacing the variables $\bfx$ (corresponding to a set $H$ of generators of $S_{\newt(f)}$) with variables corresponding to the set $\xi_{(m,g)}(H)$, which is a generating set of $S_{\newt(\mut_m^g f)}$, and by replacing each deformation parameter $t_s\in \bft_f$ with $t_{\psi_{(m,g)}s}\in \bft_{\mut_m^gf}$. This means that a mutation also transforms the deformation equations (and their linear relations) of $(X_{\newt(f)},\partial X_{\newt(f)})$ into deformation equations (and their linear relations) of $(X_{\newt(\mut_m^g f)}, \partial X_{\newt(\mut_m^g f)})$. Precise statements are given and proven in Theorems \ref{th main 1} and \ref{th main 2}.

The second main step is to show that a Laurent polynomial $f$ is mutation equivalent to a Laurent polynomial $g$ for which there exists a lattice point $v\in \newt(g)$ such that $m(v)\leq 0$ for all $m\in \cM(g)$. This is proved constructively in Theorem \ref{th smoothable}. 

These two steps show that $(X_{\newt(f)}, \partial X_{\newt(f)})$ is unobstructed in $\bft_f$, since $$(X_{\newt(g)}, \partial X_{\newt(g)})$$ is unobstructed in $\bft_g$ for cohomological reasons (see Theorem \ref{th main3}). Moreover, the smoothness of the general fibre is preserved under mutations, and thus the general fibre is smooth if and only if $f$ is $0$-mutable (see Theorem \ref{th smooth}).

In Section \ref{sec tangspace} we show that 
 $$T^1_{\paar}=\bigoplus_{k\in \NN}T^1_{\paar}(-kR^*)\bigoplus_{m\in \cT} T^1_{\paar}(-m),$$
where $\dim_\CC T^1_{\paar}(-m)=1$, if $m\in \cT$.  Applying these cohomology computations to the deformation family of $(X_{\newt(f)}, \partial X_{\newt(f)})$ over $\CC[[\bft_f]]$, we see that restricting this family to $\CC[[\bfTT_f]]$ establishes all the claims in Theorem \ref{thmm}, except for the final one, which is proved in Theorem \ref{th dis kon}.

In \cite{Fil23}, we showed that the deformations induced by the elements of $$\bigoplus_{k \in \mathbb{N}} T^1_{\paar}(-kR^*),$$ are related to Minkowski decompositions of $P$, where $X=X_P$. On the other hand, by Theorem \ref{thmm} of this paper, we see that the deformations induced by the elements of $\bigoplus_{m \in \cT} T^1_{\paar}(-m)$ are related to Laurent polynomials whose Newton polynomial is $P$. We show how to combine these two types of deformations in Section \ref{comp sec} (see Corollary \ref{cor sec 8}). This provides strong evidence for \cite[Conjecture~A]{CFP22}, and additionally suggests that all components of the miniversal deformation space of the three-dimensional affine toric Gorenstein pair $(X_P, \partial X_P)$ correspond to maximally mutable Laurent polynomials $f$ with $\newt(f)=P$ (see Conjecture \ref{sing comp conj}). 
We conclude the paper by outlining
how we expect the methods developed here to lead to a construction of Fano manifolds with a very
ample anticanonical bundle (see Section \ref{zad sec}).

\subsection*{Acknowledgement}
I am grateful to Alessio Corti for many insightful discussions.

\section{Preliminaries}\label{mut def}

\subsection{The setup}\label{sub the setup}
We fix the ground field to be $\mathbb{C}$, an algebraically closed field of characteristic zero.
Let $P$ be a lattice polytope with vertices $v^1, \dots, v^p$ in a lattice $N$. By embedding $P$ at height $1$, we obtain the rational polyhedral cone $$\sigma:=\left\{\sum_{i=1}^p\lambda_ia^i\mid \lambda_i\in \RR_{\geq 0}\right \}\subset (N \oplus \mathbb{Z}) \otimes_{\mathbb{Z}} \mathbb{R},$$
where $a^i = (v^i,1)\in N\oplus \ZZ$ for $i = 1, \dots, p$.  
Let $M$ be the dual lattice of $N$, and consider the monoid  
$$S_P := \sigma^\vee \cap (M \oplus \mathbb{Z}),$$
where  
$$\sigma^\vee := \{ r \in (M \oplus \mathbb{Z}) \otimes_{\mathbb{Z}} \mathbb{R} \mid \langle n, r \rangle \geq 0 \text{ for all $n\in \sigma$}\}$$
is the dual cone of $\sigma$.  
Every affine Gorenstein toric variety is isomorphic to  
$$X := X_P := \operatorname{Spec} \mathbb{C}[S_P]$$  
for some lattice polytope $P$. 
We set  
$$\widetilde{M} := M \oplus \mathbb{Z}, \quad \tilde{N} := N \oplus \mathbb{Z},$$  
and denote the projections  
$$
\pi_M: \widetilde{M} \to M, \quad \pi_{\mathbb{Z}}: \widetilde{M} \to \mathbb{Z},
$$
which we will use throughout the paper.

\begin{definition}\label{d:eta(c)}
For a polytope $Q\subset N_\RR:=N\otimes_\ZZ \RR$ and $c\in M_\RR:=M\otimes_\ZZ \RR$ we choose a vertex $v_Q(c)$ of $Q$ where $\lan c,\cdot \ran$ becomes minimal.
For $c\in M$ we define $$\eta_Q(c) :=-\min_{v\in Q}\langle v,c\rangle=
-\langle v_Q(c),c\rangle.$$ 
\end{definition}
The Hilbert basis of $S_P = \sigma^{\vee} \cap (M \oplus \mathbb{Z})$ is given by  
\begin{equation}\label{eg hilbbas}
H_P := \big\{ s_1 = (c_1, \eta_P(c_1)), \dots, s_r = (c_r, \eta_P(c_r)), R^* \big\},
\end{equation}
where $R^*:=(\underline{0},1)$ is the \emph{Gorenstein degree} and the elements $c_i \in M$ are uniquely determined.

\subsection{Mutations}\label{subsec mut 2}

\begin{definition}\label{norm laurent}
A Laurent polynomial $f=\sum_v a_v \chi^v \in \CC[N]$ is called \emph{normalized} if $a_v = 1$ for every vertex $v$ of its Newton polytope $\newt(f)$.
\end{definition}

All Laurent polynomials in this paper are assumed to be normalized. 
The element $$m=(\pi_M(m),\pi_\ZZ(m)) \in \widetilde{M}=M\oplus \ZZ$$ defines an affine function $\varphi_m$ on $N$ (and thus on $\newt(f)$) by  
\begin{equation}\label{eq varphim}
\varphi_m(n) := \langle \pi_M(m), n \rangle + \pi_\ZZ(m).
\end{equation}

For an element $m\in \tM$ and $k\in \ZZ$ we denote the affine hyperplanes 
$$
(m=k):=\{a\in \tN~|~\lan a,m\ran=k\},$$ 
$$(\pi_M(m)=k):=\{a\in N~|~\lan a,\pi_M(m)\ran=k\}.$$

\begin{definition}\label{def mmut}
For $g \in \mathbb{C}[N]$ and $m\in \tM$ such that $\newt(g)\subset (\pi_M(m)=0)$, we say that $f\in \CC[N]$ is \emph{$(m,g)$-mutable} if it can be written as  
\begin{equation}\label{eq mmutable}
f = \sum_{i\in \ZZ} f_i,
\qquad
\text{where}
\qquad
f_i \in \mathbb{C}[(\varphi_m = i) \cap N] \subset \CC[N],
\end{equation}
such that, for $i \in \NN$, the quotient $\frac{f_i}{g^i}$ is a Laurent polynomial (that is, $f_i = h_i g^i$ for some Laurent polynomial $h_i$).

\end{definition}  

\begin{definition}\label{def mutation}
A \emph{mutation of an $(m,g)$-mutable Laurent polynomial $f$}, with respect to the chosen pair $(m,g)$, is the Laurent polynomial  
\begin{equation}\label{eq mut h}
\mut^{g}_{m} f := \sum_{i\in \mathbb{Z}} \frac{f_i}{g^i}.
\end{equation}
\end{definition}

\begin{example}\label{ex 1}
Let  
$$f=1+2y+y^2+xy^2,~~~~m=(0,2,-3),~~~~g=1+x.$$  
We compute  
$$\varphi_m(0,0)=-3,~~~\varphi_m(0,1)=-1,~~~\varphi_m(0,2)=\varphi_m(1,2)=1.$$  
The two polytopes in Figure~\ref{fig1} represent the Newton polytopes of $f$ and its mutation, given by  
$$\mut^{g}_{m}f=1+3x+3x^2+x^3+2y+2xy+y^2.$$
\end{example}

\begin{figure}[htb]
\centering
\begin{tikzpicture}[scale=0.8]
\draw[thick, color=black]  
  (0,0) -- (0,2) -- (1,2) -- cycle;
\fill[thick, color=black]
  (0,0) circle (2.5pt) 
  (0,1) circle (2.5pt)
  (0,2) circle (2.5pt)
  (1,2) circle (2.5pt);
\node[anchor=south west] at (0,0) {$1$};
\node[anchor=south west] at (0,1) {$2$};
\node[anchor=south west] at (0,2) {$1$};
\node[anchor=south west] at (1,2) {$1$};
\end{tikzpicture}
\quad
\begin{tikzpicture}[scale=0.8]
\draw[thick, color=black]  
  (0,0) -- (3,0) -- (0,2) -- cycle;
\fill[thick, color=black]
  (0,0) circle (2.5pt) 
  (0,1) circle (2.5pt)
  (0,2) circle (2.5pt)
  (1,0) circle (2.5pt)
  (1,1) circle (2.5pt)
  (2,0) circle (2.5pt) 
  (3,0) circle (2.5pt);
\node[anchor=south west] at (0,0) {$1$};
\node[anchor=south west] at (0,1) {$2$};
\node[anchor=south west] at (0,2) {$1$};
\node[anchor=south west] at (1,0) {$3$};
\node[anchor=south west] at (2,0) {$3$};
\node[anchor=south west] at (3,0) {$1$};
\node[anchor=south west] at (1,1) {$2$};
\end{tikzpicture}
\caption{Newton polytopes of $f$ (left) and $\mut^{g}_{m} f$ (right).}
\label{fig1}
\end{figure}

\subsection{Deformation pairs}

\begin{definition}\label{def def pair}
Every pair $(m,Q)$, with $m\in \tM$ and $Q\subset (\pi_M(m)=0)\subset N$ a lattice polytope, is called a \emph{deformation pair} of $P$ if, for every $i \in \mathbb{N}$ such that $P \cap (\varphi_m = i)$ is nonempty, the polytope $iQ$ is a Minkowski summand of $P \cap (\varphi_m = i)$. 
\end{definition}

\begin{lemma}\label{lem mgex}
If $(m,Q)$ is a deformation pair of $P$, then there exist Laurent polynomials $f$ and $g$ such that $\newt(f)=P$ and $\newt(g)=Q$, and $f$ is $(m,g)$-mutable. Conversely, if $f$ is a Laurent polynomial that is $(m,g)$-mutable, then $(m,\newt(g))$ is a deformation pair of $\newt(f)$.
\end{lemma}
\begin{proof}
If $(m,Q)$ is a deformation pair, we can choose arbitrary $g$ with $\newt(g)= Q$ and $f=\sum_{i\in \ZZ}f_i$ such that $f_i=g^ig_i'$ for some Laurent polynomials $g'_i$, where each $f_i\subset \CC[(\varphi_m=i)\cap N]$ and $\newt(f)=P$. The other direction follows immediately by definition.
\end{proof}

\begin{definition}\label{def 210}
For a deformation pair $(m,Q)$ of $P$, we choose Laurent polynomials $f$ and $g$ such that $f$ is $(m,g)$-mutable, $\newt(f) = P$, and $\newt(g)=Q$. We then define the polytope  
\begin{equation}\label{eq pQm}
P_{(m,Q)}:=\newt(\mut^{g}_{m}f),
\end{equation}
which we call a \emph{mutation of $P$} by $(m,Q)$.
\end{definition}

\begin{remark}
For a different choice of $f$ and $g$, with $\newt(f)=P$, $\newt(g)=Q$, and $f$ being $(m,g)$-mutable, we obtain the same polytope $P_{(m,Q)}$. 
\end{remark}

For $m\in \tM$ and a polytope $Q\subset N_\RR$, we define the map 
\begin{equation}\label{eq ximq}
\xi_{(m,Q)}:\tM\to \tM,~~~ \xi_{(m,Q)}(h):=h-\left(\eta_Q(\pi_M(h))\right) m.
\end{equation}
Note that this map is piecewise linear and we will only use it when $(m,Q)$ is a deformation pair of $P$. 

\begin{lemma}\label{st lem}
    Let $(m,Q)$ be a deformation pair of $P$. The map $\xi_{(m,Q)}$ maps the monoid $S_P$ bijectively into the monoid $S_{P_{(m,Q)}}$, with inverse equal to $\xi_{(-m,Q)}.$
\end{lemma}
\begin{proof}
Let $(c, \eta_P(c)) \in \partial S_P$ for some $c\in \tM$. In particular, we know that 
\begin{equation}\label{eq vic}
\langle (v, 1), (c, \eta_P(c)) \rangle \geq 0 \quad \text{for all vertices } v \in P.
\end{equation} 
Let us show that 
$$
\langle (w, 1), \xi_{(m,Q)}((c, \eta_P(c))) \rangle \geq 0 \quad \text{for all vertices } w \text{ of } P_{(m,Q)}.
$$ 
Since 
$$
\xi_{(m,Q)}((c, \eta_P(c))) = (c, \eta_P(c)) - \eta_Q(c) \cdot m,
$$
we can easily verify that the claim holds for all $w \in P_{(m,Q)}\cap (\varphi_m = i)$, for every $i \in \mathbb{Z}$. Indeed, we need to show that
\begin{equation}\label{eq vic 2}
\langle w, c \rangle + \eta_P(c) - i \cdot \eta_Q(c) \geq 0
\end{equation}
for every $w \in P_{(m,Q)}\cap (\varphi_m = i)$. This follows since $w + i q \in P \cap (\varphi_m = i)$ for every $q \in Q$. Choose $q \in Q$ such that $\langle \cdot, c \rangle$ achieves its minimum on $Q$ at $q$. Then the inequality \eqref{eq vic 2} follows from \eqref{eq vic}. 
\end{proof}

\begin{example}\label{ex 2}
Let $P=\conv\{(0,0),(0,2),(1,2)\}$, $m=(0,2,-3)$, and $Q=\conv\{(0,0),(1,0)\}$. 
Note that $P=\newt(f)$ and $Q= \newt(g)$ from Example \ref{ex 1}, and thus  
$$P_{(m,Q)}=\newt(\mut^{g}_{m}f)=\conv\{(0,0),(3,0),(0,2)\}.$$  
The Hilbert basis of the semigroup $S_P$ is  
$$H_P=\{z_1,~z_4,~z_5,~z_6,~R^*=(0,0,1)\},$$
where 
$$z_1=(-2,1,0),~z_4=(-1,0,1),~z_5=(0,-1,2),~z_6=(1,0,0).$$
The Hilbert basis of the semigroup $S_{P_{(m,Q)}}$ is  
$$H_{P_{(m,Q)}}=\{s_1,s_2,s_3,s_4,s_5,s_6,R^*\},$$  
where  
$$s_1=(-2,-3,6),~s_2=(-1,-1,3),~s_3=(0,1,0),$$
$$s_4=(-1,-2,4),~s_5=(0,-1,2),~s_6=(1,0,0).$$  
The Hilbert bases of both are illustrated in the first and third images of Figure \ref{fig 2}, where the numbers indicate the third coordinate of $s_i$ and $z_i$, respectively.  

We see that  
$$\xi_{(m,Q)}(z_4)=z_4-m=(-1,-2,4)\in H_{P_{(m,Q)}}.$$  
Likewise, we see that  
$$\xi_{(m,Q)}(z_i) = s_i \quad \text{for all } i \in \{1, \dots, 6\}.$$
Note that the set $\{z_1,\dots,z_6,R^*\}$ is only a generating set of $S_{P}$, and not its Hilbert basis, which is the minimal generating set.
\end{example}

\begin{figure}[htb]
\begin{tikzpicture}[scale=0.8]
\fill[thick,  color=black]
  (0,0) circle (1.5pt) 
  (-1,0) circle (1.5pt)
  (0,-1) circle (1.5pt)
  (-2,1) circle (1.5pt)
    (1,0) circle (1.5pt);
  \draw[->]        (0,0)   -- (-2,1);
      \draw[->]        (0,0)   -- (0,-1);
          \draw[->]        (0,0)   -- (1,0);
     \node[anchor=south west] at (0,0) {$1$};
    \node[anchor=south west] at (-1,0) {$1$};
     \node[anchor=south west] at (0,-1) {$2$};
       \node[anchor=south west] at (-2,1) {$0$};
        \node[anchor=south west] at (1,0) {$0$};
\end{tikzpicture}
\begin{tikzpicture}[scale=0.8]
%\draw[thick,  color=black]  
%  (0,0) -- (4,0) -- (0,5) -- cycle;
\fill[thick,  color=black]
  (0,0) circle (1.5pt) 
  (-1,0) circle (1.5pt)
  (0,-1) circle (1.5pt)
   (0,1) circle (1.5pt)
  (-1,1) circle (1.5pt)
  (-2,1) circle (1.5pt)
    (1,0) circle (1.5pt);
  \draw[->]        (0,0)   -- (-2,1);
      \draw[->]        (0,0)   -- (0,-1);
          \draw[->]        (0,0)   -- (1,0);
     \node[anchor=south west] at (0,0) {$R^*$};
    \node[anchor=south west] at (-1,0) {$z_4$};
     \node[anchor=south west] at (0,1) {$z_3$};
     \node[anchor=south west] at (0,-1) {$z_5$};
       \node[anchor=south west] at (-1,1) {$z_2$};
       \node[anchor=south west] at (-2,1) {$z_1$};
        \node[anchor=south west] at (1,0) {$z_6$};
\end{tikzpicture}
\\
\begin{tikzpicture}[scale=0.8]
%\draw[thick,  color=black]  
%  (0,0) -- (4,0) -- (0,5) -- cycle;
\fill[thick,  color=black]
  (0,0) circle (1.5pt) 
  (1,0) circle (1.5pt)
  (0,1) circle (1.5pt)
    (-1,-2) circle (1.5pt)
  (-1,-1) circle (1.5pt)
  (-2,-3) circle (1.5pt)
    (0,-1) circle (1.5pt);
  \draw[->]        (0,0)   -- (-2,-3);
      \draw[->]        (0,0)   -- (0,1);
          \draw[->]        (0,0)   -- (1,0);
     \node[anchor=south west] at (0,0) {$1$};
     \node[anchor=south west] at (0,1) {$0$};
     \node[anchor=south west] at (0,-1) {$2$};
       \node[anchor=south west] at (-1,-2) {$4$};
       \node[anchor=south west] at (-2,-3) {$6$};
        \node[anchor=south west] at (-1,-1) {$3$};
        \node[anchor=south west] at (1,0) {$0$};
\end{tikzpicture}
~~~
\begin{tikzpicture}[scale=0.8]
\fill[thick,  color=black]
  (0,0) circle (1.5pt) 
  (1,0) circle (1.5pt)
  (0,1) circle (1.5pt)
    (-1,-2) circle (1.5pt)
  (-1,-1) circle (1.5pt)
  (-2,-3) circle (1.5pt)
    (0,-1) circle (1.5pt);
  \draw[->]        (0,0)   -- (-2,-3);
      \draw[->]        (0,0)   -- (0,1);
          \draw[->]        (0,0)   -- (1,0);
     \node[anchor=south west] at (0,0) {$R^*$};
     \node[anchor=south west] at (0,1) {$s_3$};
     \node[anchor=south west] at (0,-1) {$s_5$};
       \node[anchor=south west] at (-1,-2) {$s_4$};
       \node[anchor=south west] at (-2,-3) {$s_1$};
        \node[anchor=south west] at (-1,-1) {$s_2$};
        \node[anchor=south west] at (1,0) {$s_6$};
\end{tikzpicture}
\caption{A generating set of $S_{\newt(f)}$ and $S_{\newt(\mut_m^gf)}$.}
\label{fig 2}
\end{figure}

\section{Constructing one-parameter deformations}\label{sec 3}

This section provides an explicit construction of a one-parameter deformation arising from a deformation pair, formulated in terms of the defining equations.

\subsection{Linear relations between the generators}\label{sub31}

Recall the Hilbert basis \eqref{eg hilbbas} of $S:=S_P=\sigma^\vee\cap \tM$. We have  
$$X_P=\operatorname{Spec} \mathbb{C}[S_P] \cong \operatorname{Spec} \mathbb{C}[u, x_1, \dots, x_r] / \mathcal{I}_P,$$  
for some ideal $\mathcal{I}_P$ (described in Proposition \ref{acf1} below), where $u$ corresponds to $R^*$ and each $x_j$ corresponds to $s_j$ for $j = 1, \dots, r$.
For every $\bfk=(k_1,...,k_r)\in \NN^r$
we denote 
$$s_\bfk:=\sum_{i=1}^rk_is_i\in S\subset \tM,~~~~~c_{\bfk}:=\sum_{i=1}^rk_ic_i\in M$$ and for a polytope $Q\subset N_\RR$ let 
$$\eta_Q(\bfk):=\sum_{i=1}^r\eta_Q(k_ic_i)-\eta_Q\left (\sum_{i=1}^rk_ic_i\right ).$$
For every element $s\in S$ we have a unique decomposition $s=\partial_P(s)+n_P(s)R^*$ with 
$$\partial_P(s)\in \partial(S):=\{s\in S~|~s-R^*\not\in S\}\text{~~~and~~~}n_P(s)\in \NN.$$ Let us denote 
\begin{equation}\label{eq chis 1}
\chi^s_P:=\bfx^{\partial_P(s)}u^{n_P(s)}.
\end{equation}

We have $s_{\bfk}=\bo_P(\bfk)+\eta_P(\bfk)R^*$ with $\bo_P(\bfk):=(c_\bfk,\eta_P(c_\bfk))\in \partial(S)$. 
In this section we simply write $\partial(\bfk)=\partial_P(\bfk)$, $\eta(\bfk)=\eta_P(\bfk)$ and $\chi^s=\chi^s_P$ for any $s\in S=S_P$.
For every $\bfk\in \NN^r$ we choose $b_i\in \NN$ such that 
$
 \bo(\bfk)=\sum_{i=1}^rb_is_i, 
$
and we define
$$\bfx^{\bfk}:=\prod_{i=1}^rx^{k_i}_i,~~~\bfx^{\bo(\bfk)}:=\prod_{i=1}^rx_i^{b_i}.$$

\begin{remark}\label{rem par remm}
For simplicity, we will also write $\partial(\bfk) = (b_1, \dots, b_r) \in \mathbb{N}^r$, and we will distinguish from context whether $\partial(\bfk)$ refers to the vector in $\mathbb{N}^r$ or to the element $\sum_{i=1}^r b_i s_i \in \tM$. 
More generally, for each element $s \in S$, we will distinguish from the context whether $\partial(s)$ refers to the tuple $(p_1, \dots, p_r) \in \NN^r$ or to the sum $\sum_{i=1}^r p_is_i \in \partial(S)$.
For example, this convention will be used in the proof of Proposition \ref{th main 2}.
\end{remark}

\begin{proposition}[{\cite[Section 5]{ACF22a}}]\label{acf1}
The binomials 
\begin{equation}\label{eq fbfk xu2}
f_\bfk(\bfx,u):=f_{\bfk,P}(\bfx,u):=\bfx^\bfk - \bfx^{\bo(\bfk)}\,u^{\eta(\bfk)}\in \CC[u,\bfx]:=\CC[u,x_1,...,x_r]
\end{equation}
generate the ideal $\cI_P$ and the module of linear relations among the $f_\bfk$, which is the kernel of the map
$$
  \textstyle
  \psi:\bigoplus_{\bfk\in \NN^r}\CC[u,x_1,\dots,x_r]e_\bfk\xrightarrow{e_\bfk\mapsto f_\bfk} \cI_P\subset \CC[u,x_1,\dots,x_r],
$$
 is spanned by $r_{\bfa,\bfk}:= e_{\bfa+\bfk} - x^\bfa e_\bfk - u^{\eta(\bfk)}e_{\bo(\bfk)+\bfa}$, for $\bfa,\bfk\in \NN^r$. 
\end{proposition}

\begin{example}\label{ex3}
 
    Let $P=\conv\{(0,0),(3,0),(0,2)\}$ with 
    $$s_1=(-2,-3,6), s_2=(-1,-1,3), s_3=(0,1,0),$$ 
    $$s_4=(-1,-2,4), s_5=(0,-1,2), s_6=(1,0,0).$$
   Note that we have already drawn the Hilbert basis of $S=S_P$, since this polytope was the polytope $P_{(m,Q)}$ in Example \ref{ex 2}.
    Let $\bfk=(0,1,0,0,0,1)\in \NN^6$ and $\bfa=(0,0,0,1,0,1)\in \NN^6.$ Then we have 
    $\partial(\bfk)=(0,0,0,0,1,0)$ and the following equations:
    $$f_{\bfk}=x_2x_6-ux_5,~~~f_{\bfa+\bfk}=x_2x_4x_6^2-ux_5^3,~~~f_{\partial(\bfk)+\bfa}=x_4x_5x_6-x_5^3,$$
    where $\partial(\bfa+\bfk)=\partial(\partial(\bfk)+\bfa)=(0,0,0,0,3,0)$.
\end{example}

\begin{lemma}\label{lem bfk}
For $\bfk\in \NN^r$ and polytopes $Q,G\in N_\RR$, the following holds:
\begin{enumerate}
%\item $i\eta_Q(\bfk)=\eta_{iQ}(\bfk)$ for any $i\in \NN$,
\item $\eta_{Q+G}(\bfk)=\eta_Q(\bfk)+\eta_G(\bfk)$,
\item $\eta_{Q}(\bfa+\bfk)=\eta_{Q}(\bfk)+\eta_{Q}(\partial(\bfk)+\bfa)$.
\end{enumerate}
\end{lemma}
\begin{proof}
The first claim follows immediately from the definition. We now prove (2): after canceling  the common terms
$$\sum_{i=1}^r(\eta_{Q}(k_ic_i)+\eta_Q(a_ic_i))$$  
on both sides, we need to prove that  
$$
-\eta_{Q} \left(\sum_{i=1}^r (a_i + k_i) c_i \right) =
-\eta_Q \left(\sum_{i=1}^r k_i c_i \right) 
+ \sum_{i=1}^r \eta_Q(b_i c_i) 
- \eta_Q \left(\sum_{i=1}^r (a_i + b_i) c_i \right).
$$

This holds because, by the definition of $\partial(\bfk)$, we see that if $b_i \neq 0$, then  
$$
v(c_i) = v\left(\sum_{i=1}^r k_i c_i\right).
$$
Consequently, we have  
$$
\sum_{i=1}^r \eta_Q(b_i c_i) = \eta_Q\left(\sum_{i=1}^r b_i c_i\right) \quad \text{and}\quad \sum_{i=1}^r k_i c_i = \sum_{i=1}^r b_i c_i.
$$
From this, the claim follows.
\end{proof}

\subsection{Formal deformations of a pair $\paar$}
Let $R$ be a local Artinian $\mathbb{C}$-algebra with residue field $R/m_R=\mathbb{C}$.
A \emph{deformation of a pair} $\paar$ over $R$ is a \emph{deformation of the closed embedding} $\partial X\hookrightarrow X$ over $R$, which is given by a compatible system of commutative diagrams:

\begin{equation}\label{eq:def_xi_n}
\xi_n\colon
\begin{tikzcd}
\cY_n \arrow[dr, "g_n"'] \arrow[r, hook] & \cX_n \arrow[d, "f_n"] \\
& \spec (R/m_R^n)
\end{tikzcd}
\end{equation}
for each $n\in \mathbb{N}$, such that:
\begin{itemize}
    \item $f_n$ and $g_n$ are flat and $\cY_n\hookrightarrow \cX_n$ is a closed embedding,
    \item $\xi_0$ is the closed embedding $\partial X\hookrightarrow X$ over $\operatorname{Spec} \mathbb{C}$,
    \item for all $n\geq 1$, $\xi_n$ induces $\xi_{n-1}$ by pullback under the natural inclusion $\operatorname{Spec} (R/m_R^{n-1}) \to \operatorname{Spec} (R/m_R^n)$.
\end{itemize}

It is straightforward to define a deformation functor $F_{\paar}$, which associates to $R$ the set of isomorphism classes of deformations of $\paar$ over $R$. The corresponding tangent space is denoted by $T^1_{\paar}$. 
If we disregard $\cY_n$ and consider only the maps $\cX_n \to \operatorname{Spec} (R/m^n_R)$ satisfying the above properties, we obtain the deformation functor $F_X$, which assigns to $R$ the set of isomorphism classes of deformations of $X$. The corresponding tangent space is denoted by $T^1_X$.

The following lemma provides the equational flatness criterion for Gorenstein affine toric varieties.
\begin{lemma}[Equational Flatness Criterion]\label{lem equat fla}
    Let $X_P = \spec \CC[\bfx,u]/(f_\bfk~|~\bfk\in \NN^r)$ be a Gorenstein affine toric variety, and let $R = \CC[[\bft]]/I=\CC[[t_1,\dots,t_n]]/I$, where $I$ is an ideal. Recall $r_{\bfa,\bfk}(\bfx,u)$ from Proposition \ref{acf1}.
    If there exist $$F_\bfk(\bfx,u,\bft) \in \CC[\bfx,u][[\bft]]$$ such that:
    \begin{itemize}
        \item $F_\bfk(\bfx,u,0) = f_\bfk(\bfx,u)$ for all $\bfk\in \NN^r$.
        \item There exist linear relations $R_{\bfa,\bfk}(\bfx,u,\bft)$ among $F_\bfk$ in $\CC[\bfx,u][[\bft]]/I$ such that
$$
R_{\bfa,\bfk}(\bfx,u,0) = r_{\bfa,\bfk}(\bfx,u) \quad \text{for all } \bfa, \bfk \in \mathbb{N}^r,
$$
    \end{itemize}
then there exists a formal deformation of $(X_P,\partial X_P)$ over $R = \CC[[\bft]]/I$ with 
$$
\cX_n = \spec \CC[\bfx,u][[\bft]] / \left(m_R^n + (F_\bfk(\bfx,u,\bft) \mid \bfk \in \NN^r) \right)
$$
and 
$$
\cY_n = \spec \CC[\bfx,u][[\bft]] / \left(m_R^n + (u, F_\bfk(\bfx,u,\bft) \mid \bfk \in \NN^r) \right).
$$
\end{lemma}
\begin{proof}
  It follows immediately from Proposition \ref{acf1}, see, e.g., \cite[Corollary 6.5]{Eis95} or \cite[Section 10.3]{JP00} for the general form of the equational flatness criterion.
\end{proof}

\begin{definition}\label{def 5757}
    We say that we have formal deformation $\{F_\bfk(\bfx,u,\bft)\mid \bfk\in \NN^r\}$ of $(X_P,\partial X_P)$ over $\CC[[\bft]]/I$ if there exists linear relations $R_{\bfa,\bfk}(\bfx,u,\bft)$ among $F_\bfk$ in $\CC[\bfx,u][[\bft]]/I$ such that the conditions of Lemma \ref{lem equat fla} are satisfied.
\end{definition}

\subsection{Deformation of equations for one-parameter deformations}\label{subsec pert}

For a given deformation pair $(m,Q)$ of $P$, we are going to explicitly describe a one-parameter deformation of $X_P$ in terms of the deforming equations $f_{\bfk}$. This is the first key result of the paper, and it plays a crucial role in proving Theorem \ref{thmm}.
 Let $t=t_{(m,Q)}$ be a deformation parameter of lattice degree $\deg(t)=m$. We define

\begin{multline}\label{eq Fk}
F_{\bfk}(\bfx,u,t):=\bfx^\bfk-\sum_{i=0}^{\eta_Q(\bfk)}{\eta_Q(\bfk)\choose i}t^i \chi_P^{s_\bfk-im}=\\
=f_\bfk-\sum_{i=1}^{\eta_Q(\bfk)}{\eta_Q(\bfk)\choose i}t^i \bfx^{\partial_P(s_\bfk-im)}u^{n_P(s_\bfk-im)},
\end{multline}
where we used the notation from \eqref{eq chis 1} and $t=t_{(m,Q)}$ is a deformation parameter that has $\tM$-degree $\deg(t)=m$. 

Note that $\chi^{s_\bfk-im}=\chi_P^{s_\bfk-im}$ is well-defined for all $i=1,\dots,\eta_Q(\bfk)$, meaning that $s_\bfk - im \in \sigma^\vee$. To prove this, it suffices to show that $s_\bfk - im$ takes nonnegative values on the generators of $\sigma$, which lie in $\sigma \cap (R^* = 1)$.  
To do so, it is enough to show that $\eta_{P_j}(\bfk) \geq j\cdot \eta_Q(\bfk)$ for all $j \in \mathbb{N}$ such that $P_j := P \cap (\varphi_m = j)$ is nonempty.
This follows from Lemma \ref{lem bfk} (1), since $jQ$ is a Minkowski summand of $P_j$ and thus $\eta_{P_j}(\bfk) \geq \eta_{jQ}(\bfk)=j\eta_Q(\bfk)$.

Let us choose $(k_{1j},...,k_{rj})\in \NN^r$, $j=1,...,\eta_Q(\bfk)$ such that 
$$\sum_{l=1}^r k_{lj} s_l = \partial(s_\bfk - jm)$$
and let $\bfk_j:=(k_{1j},...,k_{rj})\in \NN^r$. We define
\begin{equation}\label{eq r one par}
R_{\bfa,\bfk}(\bfx,u,t):=F_{\bfa+\bfk}-\bfx^\bfa F_{\bfk}-u^{\eta_P(\bfk)}F_{\partial(\bfk)+\bfa}-\sum_{j=1}^{\eta_Q(\bfk)}u^{\eta_P(\bfk_j)}{\eta_Q(\bfk) \choose j}t^jF_{\bfk_j+\bfa}.
\end{equation}
We see that $F_\bfk(\bfx,u,0)=f_\bfk(\bfx,u)$ and $R_{\bfa,\bfk}(\bfx,u,0)=r_{\bfa,\bfk}(\bfx,u)$. 

\begin{remark}
As mentioned in Remark~\ref{rem par remm}, we can also write $F_{\partial(s_\bfk - jm) + \bfa}$ in place of $F_{\bfk_j + \bfa}$, where $\partial(s_\bfk - jm) = \bfk_j \in \mathbb{N}^r$ (and not $\partial(s_\bfk - jm) \in \partial(S_P)$).
\end{remark}

\begin{proposition}\label{pro main pro}
$R_{\bfa,\bfk}$ is a linear relation (among $F_{\bfk}$, for all $\bfk\in \NN^r$). 
\end{proposition}
\begin{proof}
To verify that $R_{\bfa,\bfk}$ is a linear relation, we compute that the term in front of $t^l$ is zero for all $l\in \NN$. This term is equal to $$-{\eta_Q(\bfa+\bfk)\choose l}\chi^{s_{\bfa+\bfk}-lm}+\bfx^\bfa{\eta_Q(\bfk)\choose l}\chi^{s_{\bfk}-lm}+u^{\eta_P(\bfk)}{\eta_Q(\partial(\bfk)+\bfa)\choose l}\chi^{s_{\partial(\bfk)+\bfa}-lm}-$$
$$-u^{\eta_P(\bfk_l)}{\eta_Q(\bfk)\choose l}(\bfx^{\bfk_l+\bfa}-\chi^{s_{\bfk_l+\bfa}})+\sum_{j=1}^{l-1}u^{\eta_P(\bfk_j)}{\eta_Q(\bfk) \choose j}{\eta_Q(\bfk_j+\bfa)\choose l-j}\chi^{s_{\bfk_j+\bfa}-(l-j)m}.$$

Note that by definition we have 
 $\bfx^\bfa\chi^{s_\bfk-lm}=u^{\eta_P(\bfk_l)}\bfx^{\bfk_l+\bfa}$ and $$\chi^{s_{\bfa+\bfk}-lm}=u^{\eta_P(\bfk)}\chi^{s_{\partial(\bfk)+\bfa}-lm}=u^{\eta_P(\bfk_j)}\chi^{s_{\bfk_j+\bfa}-(l-j)m},$$ for all $j=1,...,l$. Thus we see that the term before $t^l$ is zero since $\eta_Q(\bfk_j+\bfa)=\eta_Q(\partial(\bfk)+\bfa)$ for every $\bfk_j$ and $${\eta_Q(\bfa+\bfk)\choose l}={\eta_Q(\bfk)\choose l}+{\eta_Q(\partial(\bfk)+\bfa)\choose l}+\sum_{j=1}^{l-1}{\eta_Q(\bfk)\choose j}{\eta_Q(\partial(\bfk)+\bfa)\choose l-j},$$
 which holds because $$(1+t)^{\eta_Q(\bfa+\bfk)}=(1+t)^{\eta_Q(\bfk)}(1+t)^{\eta_Q(\partial(\bfk)+\bfa)}$$ by Lemma \ref{lem bfk}.
\end{proof}

By the equational flatness criterion we get the following corollary.

\begin{corollary}\label{cor 3.3}
The map
$$\spec \CC[\bfx,u,t]/(F_\bfk(\bfx,u,t)~|~\bfk\in \NN^r)\to \spec \CC[[t]]$$
is flat and has fibre over $0$ equal to $X_P$. By Lemma \ref{lem equat fla} we also get a formal (one-parameter) deformation of $(X_P,\partial X_P)$ over $\CC[[t]]$, corresponding to a deformation pair $(m,Q)$ of $P$. 
\end{corollary}

\begin{example}\label{ex 3.6}
    Here we continue with Example \ref{ex3} with $$P=\conv\{(0,0),(3,0),(0,2)\}$$ and let $(-m,Q)$ be a deformation pair of $P$ with $m=(0,2,-3)$ and $Q=\conv\{(0,0),(1,0)\}$. We use $(-m,Q)$ to be consistent with notation in Example \ref{ex 2}.
    We have $$F_\bfk=f_\bfk-tx_3,~~~F_{\bfa+\bfk}=f_{\bfa+\bfk}-t^2x_3u-2tx_5u^2,~~~F_{\partial(\bfk)+\bfa}=f_{\partial(\bfk)+\bfa}-tx_5u
    $$
    and thus 
    $$R_{\bfa,\bfk}=F_{\bfa+\bfk}-\bfx^\bfa F_{\bfk}-uF_{\partial(\bfk)+\bfa}-tF_{\bfk_1+\bfa},$$
    where $t=t_{(-m,Q)}$ and $\bfk_1=(0,0,1,0,0,0)$ and thus $F_{\bfk_1+\bfa}=x_3x_4x_6-u^2x_5-tx_3u$. Note that $R_{\bfa,\bfk}$ is indeed a linear relation, since the term in front of $t^2$ is $-t^2x_3u+t(tx_3u)=0$ and the term in front of $t$ is $-2tx_5u^2+\bfx^\bfa tx_3+utx_5u-t(x_3x_4x_6-u^2x_5)=0$, since $\bfx^\bfa=x_4x_6$.
\end{example}

\section{Mutations of Laurent polynomials in two variables}\label{sec 4}
This section is devoted to the study of mutations of Laurent polynomials in two variables, which play a key role in constructing the deformation families developed in the subsequent sections.

\begin{lemma}\label{lem pocor}
Let $A$ be an integral domain. 
If the polynomial 
$$
a_0+\cdots+a_nx^n\in A[x]
$$ 
is divisible by $p^m$, with $p\in A[x]$ and $m\in \NN$, then the polynomial 
$$
\sum_{r=0}^n {z_1+z_2r\choose k} a_r x^r\in A[x]
$$
is divisible by $p^{m-k}$ for any $z_1,z_2\in \ZZ$ and any $k\in \NN$, provided that $k\leq m$.  
\end{lemma}
\begin{proof}
    By induction and Pascal's identity, it suffices to prove the claim for $z_1=0$. Define
$$
f(x):=a_0+\cdots+a_nx^n \quad \text{and} \quad g(x):=f(x^{z_2}).
$$
Since $f(x)$ is divisible by $p^{m}(x)$, the $k$-th derivative $g^{(k)}(x)$ is divisible by $p^{m-k}(x)$. Consequently, 
$$
\frac{x^k}{k!}g^{(k)}(x)=\sum_{r=0}^n{z_2r\choose k}a_rx^{rz_2}.
$$
is divisible by $p^{m-k}(x^{z_2})$ in $A[x]$. The change of variable $y = x^{z_2}$  shows that
$$
\sum_{r=0}^n {z_2r\choose k} a_r y^r\in A[y]
$$
is divisible by $p^{m-k}(y)$. This completes the proof.
\end{proof}

\begin{definition}\label{def m mut}
Let  $m \in \tM\cong \ZZ^3$, $m\neq kR^*$ for any $k\in \ZZ$. A Laurent polynomial $f \in \CC[N]\cong \CC[\ZZ^2]$ is called \emph{$m$-mutable} if it is $(m,g)$-mutable with 
$\newt(g)$
a line segment of lattice length $1$.  
\end{definition}

\begin{remark}
We assume that all Laurent polynomials are normalized; in particular, $g$ is assumed to take the value $1$ at both vertices of $\newt(g)$.
\end{remark}

\begin{definition}\label{def prva not def}
Let $m\in \tM\cong \ZZ^2$. If $g$ is a Laurent polynomial with $\newt(g)\subset (\pi_M(m)=0)$
a line segment of lattice length $1$, we define the map
$$
\psi_{(m,g)}: \tM \to \tM
$$
by
\begin{equation}\label{eq psimr}
\psi_{(m,g)}(r) := 
\begin{cases}
r+\left(\eta_{\newt(g)}(\pi_M(-r))\right)m & \text{if } \pi_M(r) \text{ and } \pi_M(m) \text{ are not collinear in } M \cong \ZZ^2, \\
r-m, & \text{otherwise}.
\end{cases}
\end{equation}
\end{definition}

\begin{proposition}\label{prop mutability}
Let $f$ be $m$-mutable. Then $f$ is $r$-mutable if and only if $\mut^g_m f$ is $\psi_{(m,g)}(r)$-mutable. 
\end{proposition}
\begin{proof}
If $f$ is both $m$-mutable and $r$-mutable, it is $(m,g)$-mutable and $(r,h)$-mutable, where $$\newt(g)\subset (\pi_M(m)=0) \quad \text{and} \quad \newt(h)\subset (\pi_M(r)=0)$$ are line segments of lattice length $1$.  
If $\pi_M(r)$ and $\pi_M(m)$ are collinear in $M$, the claim follows immediately from the definition.

Assume that $\pi_M(r)$ and $\pi_M(m)$ are not collinear in $M$. Suppose first that $\eta_{\newt(g)}(\pi_M(-r))=0$ and thus $\psi_{(m,g)}(r)=r$.  
For any Laurent polynomial $q \in \CC[N]$, and for $i,j \in \ZZ$, we denote by $q_{i,j}$ the restriction of $q$ to the subset
$
\newt(q) \cap (\varphi_m = i) \cap (\varphi_r = j).
$
Let $n:=\max_{v \in \newt(f)} \varphi_r(v)$, and note that 
$$
\max_{v \in \newt(g)} \varphi_r(v) = \eta_{\newt(g)}(\pi_M(-r)) = 0.
$$  
Moreover, we denote $v_{i,j} \in N_\RR$ to be the element such that $\varphi_m(v_{i,j})=i$ and $\varphi_r(v_{i,j})=j$.

Thus, we have
\begin{equation}\label{eq mut1}
(\mut_m^g f)_{i,n-j}=\sum_{\substack{k+l=j}} f_{i,n-k} \left( g^{-i} \right)_{0,-l} = \sum_{\substack{k+l=j}} f_{i,k} \chi^{v_{0,-l}} {-i\choose p},
\end{equation}
where $l = z \cdot p$ with
\begin{equation}\label{eq zz}
 z = \max_{v\in \newt(g)} \varphi_r(v) - \min_{v\in \newt(g)} \varphi_r(v)=- \min_{v\in \newt(g)} \varphi_r(v), 
 \end{equation}
and $\chi^{v_{0,-l}}$ is set to $0$ if $v_{0,-l} \not\in N$, and otherwise it represents the variable corresponding to $v_{0,-l} \in N$.

If $f$ is $(r,h)$-mutable, with $\newt(h)$ a line segment of lattice length $1$, then $\sum_{i\in \ZZ}f_{i,n-k}$ is divisible by $h^{n-k}$, and by applying Lemma \ref{lem pocor}, it follows that 
$$\sum_{i\in \ZZ} f_{i,n-k}\chi^{v_{0,-l}}{-i\choose p}$$  
is divisible by $h^{n-k-p}$. Thus, equation \eqref{eq mut1} implies that $\sum_{i\in \ZZ}(\mut_m^gf)_{i,n-j}$ is divisible by $h^{n-j}$, since $p \leq l$ and we are summing over $k+l=j$.
This implies that $\mut_m^gf$ is $\psi_{(m,g)}(r)=r$-mutable, which is what we wanted to show.

If $\eta_{\newt(g)}(\pi_M(-r)) \neq 0$, we translate $\newt(g)$ to a polytope $\newt(\tilde{g}) \subset (\pi_M(m)=0)$ satisfying $$\eta_{\newt(\tilde{g})}(\pi_M(-r))=0.$$ The coefficient of $\mut_m^{\tilde{g}} f$ in front of $\chi^{v_{i,j}}$ equals the coefficient of $\mut_m^{g} f$ in front of $\chi^{w_{i,j}}$, where $w_{i,j}\in N_\RR$ satisfies 
$
\varphi_m(w_{i,j})=i$ and $\varphi_{\tilde{r}}(w_{i,j})=j$
with 
$
\tilde{r}=r+\eta_{\newt(g)}(\pi_M(-r))\,m.
$
Thus, we conclude that $\mut_m^g f$ is $\psi_{(m,g)}(r)$-mutable.

The converse follows immediately from the above proof, since
$$\mut_{-m}^g\left( \mut_m^gf \right)=f \quad \text{and}\quad \psi_{(-m,g)}(\psi_{(m,g)}(r))=r.$$
This completes the proof.
\end{proof}

\begin{example}
Let $f=1+3x+3x^2+x^3+2y+2xy+y^2$, $g=1+y$, $r=(0,-1,3)$, and $m=(-2,0,2)$. We see that $f$ is both $m$-mutable and $r$-mutable, and moreover, $\mut_m^gf$ is $\psi_{(m,g)}(r)=r$-mutable, since 
$$\mut_m^gf=(1+x)^3+(1+x)2y(1+2x)+3x^2y^2+6x^3y^2+4x^3y^3+x^3y^4.$$
$f$ is the first Laurent polynomial presented in Figure \ref{fig 3}, and $f_1:=\mut_m^gf$ is the second. %Laurent polynomial  presented in Figure \ref{fig 3} and 
Note that in this example $z$ from \eqref{eq zz} equals $1$. Let $h=1+x$ and $s=(0,-1,2)$. Then 
$$\mut_s^hf_1=1+x+2xy+4x^2y+3x^2y^2+6x^3y^2+4x^3y^3+4x^4y^3+x^3y^4+2x^4y^4+x^5y^4,$$
which is the third Laurent polynomial presented in Figure \ref{fig 3}. Note that $f_1$ is $s$-mutable and $-m=(2,0,-2)$-mutable. $\mut_s^hf_1$ is indeed $\psi_{(s,h)}(-m)=-m+\left(\eta_{\newt(h)}(\pi_M(m))\right)s=-m+2s=(2,-2,2)$-mutable, since $$\frac{1+2z+3z^2+4z^3+2z^4}{(1+z)^2}=1+2z^2,$$
where we set $z=xy$.
\begin{figure}[htb]
\begin{tikzpicture}[scale=0.8]
\draw[thick,  color=black]  
  (0,0) -- (3,0) -- (0,2) -- cycle;
\fill[thick,  color=black]
  (0,0) circle (2.5pt) 
  (0,1) circle (2.5pt)
  (0,2) circle (2.5pt)
  (1,0) circle (2.5pt)
  (1,1) circle (2.5pt)
  (2,0) circle (2.5pt) 
 (3,0) circle (2.5pt);
     \node[anchor=south west] at (0,0) {$1$};
    \node[anchor=south west] at (1,0) {$3$};
    \node[anchor=south west] at (2,0) {$3$};
    \node[anchor=south west] at (3,0) {$1$};
     \node[anchor=south west] at (0,1) {$2$};
    \node[anchor=south west] at (0,2) {$1$};
    \node[anchor=south west] at (1,1) {$2$};
\end{tikzpicture}
~~~~~~~~~
\begin{tikzpicture}[scale=0.8]
\draw[thick,  color=black]  
  (0,0) -- (3,0) -- (3,4) -- cycle;
\fill[thick,  color=black]
  (0,0) circle (2.5pt) 
  (1,0) circle (2.5pt)
  (1,1) circle (2.5pt)
  (2,0) circle (2.5pt) 
  (2,1) circle (2.5pt) 
  (2,2) circle (2.5pt) 
 (3,0) circle (2.5pt)
 (3,1) circle (2.5pt) 
 (3,2) circle (2.5pt) 
 (3,3) circle (2.5pt) 
 (3,4) circle (2.5pt) ;
     \node[anchor=south west] at (0,0) {$1$};
    \node[anchor=south west] at (1,0) {$3$};
    \node[anchor=south west] at (2,0) {$3$};
    \node[anchor=south west] at (3,0) {$1$};
     \node[anchor=south west] at (1,1) {$2$};
    \node[anchor=south west] at (2,1) {$6$};
    \node[anchor=south west] at (3,1) {$4$};
    \node[anchor=south west] at (2,2) {$3$};
    \node[anchor=south west] at (3,2) {$6$};
    \node[anchor=south west] at (3,3) {$4$};
    \node[anchor=south west] at (3,4) {$1$};
\end{tikzpicture}
~~~~~~~~~~
\begin{tikzpicture}[scale=0.8]
\draw[thick,  color=black]  
  (0,0) -- (1,0) -- (5,4) -- (3,4) --cycle;
\fill[thick,  color=black]
  (0,0) circle (2.5pt) 
  (1,0) circle (2.5pt)
  (1,1) circle (2.5pt)
  (2,1) circle (2.5pt) 
  (2,2) circle (2.5pt)  
 (3,2) circle (2.5pt) 
 (3,3) circle (2.5pt) 
 (3,4) circle (2.5pt)
(4,3) circle (2.5pt)
(4,4) circle (2.5pt)
(5,4) circle (2.5pt);
     \node[anchor=south west] at (0,0) {$1$};
    \node[anchor=south west] at (1,0) {$1$};
     \node[anchor=south west] at (1,1) {$2$};
    \node[anchor=south west] at (2,1) {$4$};
    \node[anchor=south west] at (2,2) {$3$};
    \node[anchor=south west] at (3,2) {$6$};
    \node[anchor=south west] at (3,3) {$4$};
    \node[anchor=south west] at (3,4) {$1$};
        \node[anchor=south west] at (4,4) {$2$};
    \node[anchor=south west] at (5,4) {$1$};
        \node[anchor=south west] at (4,3) {$4$};
\end{tikzpicture}
\caption{The Laurent polynomials $f$, $f_1:=\mut_m^gf$ and $\mut_s^hf_1$.}
\label{fig 3}
\end{figure}
\end{example}

\section{The tangent space of the deformation functor}\label{sec tangspace}

 The tangent space $T^1_X$ decomposes according to the grading by $\tM$ ($T^1_X=\bigoplus_{r\in \tM}T^1_X(-r)$) and it has been extensively studied through various convex-geometric descriptions depending on the polytope defining the affine Gorenstein toric variety (see \cite{Alt94}, \cite{Alt97b}, \cite{AS98}, and \cite{Fil18}).

Furthermore, results from \cite{CFP22} and \cite{Fil23} indicate that it is often more natural to consider deformations of the pair $\paar$ rather than those of $X$ alone.
 In our case we have
$$X=\spec \CC[S]=\spec \CC[\bfx,u]/\cI_P\quad \text{and} \quad \partial X=\spec \CC[\partial S]=\CC[\bfx,u]/(\cI_P,u).$$

Since $\partial X\hookrightarrow X$ is a regular embedding, we can apply the results from \cite{CFGK17} to describe the module $T^1_{\paar}$. We define $A=\CC[\bfx,u]/\cI_P$ and $A':=A/(u)$. 
To study deformations of a pair $\paar$, we use the following exact sequence (see, e.g., \cite[Equation 11]{CFGK17}):
\begin{equation}\label{eq varphi1}
0\to T_{\partial X}\to T_{X|\partial X}\xrightarrow{\varphi} N_{\partial X|X}\xrightarrow{\varphi_1} T^1_{(X,\partial X)}\to T^1_X\to 0,
\end{equation}
where $T_{X|\partial X} = \Der_\CC(A,A) \otimes A'$ and $N_{\partial X|X} = \Hom_{A'}((u)/(u)^2,A')$. 

The following lemma determines which derivations in $T_{X|\partial X}$ map to nonzero elements in $N_{\partial X|X}$ under the map $\varphi$. This analysis is crucial for computing the dimension of the module $T^1_{\paar}$.

We denote $\eta=\eta_P$ and $\partial=\partial_P$. Let $s_i=(c_i,\eta(c_i))\in H_P$ (for some $i=1,...,r$) be an element of the Hilbert basis such that the minimum of $\lan c_i,\cdot \ran$ on $P$ is achieved only at one vertex $v_P(c_i)$ of $P$. We fix such $i$ and for every $j=1,...,r$ we define  
$$n_j:=\sum_{z=0}^{\infty}\eta(\partial(\bfe_j+z\bfe_i)+\bfe_i),$$
where $\bfe_j$ denotes the $j$-th basis vector of $\mathbb{N}^r$.

Note that by definition there exists $z\in \NN$ such that $\eta(\partial(\bfe_j+n\bfe_i)+\bfe_i)=0$ for every $n\geq z$ and thus the above sum is finite.

\begin{lemma}\label{lem 51}
It holds that
$$
D_i := \sum_{j=1}^r A_j \frac{\partial}{\partial x_j} + x_i\frac{\partial}{\partial u}\in \Der_\CC(A,A), 
\quad\text{where}\quad 
A_j = n_j \bfx^{\partial(\bfe_i+\bfe_j)} u^{\eta(\bfe_i+\bfe_j)-1}.
$$
Note that $A_j$ is well-defined, since if $\eta(\bfe_i+\bfe_j)=0$, then $n_j=0$.
\end{lemma}

\begin{proof}
We need to check that $D_i(f_\bfk)\in \cI_P$ for all $f_\bfk\in \cI_P$. It holds that  
$$
\begin{aligned}
D_i(\bfx^\bfk - \bfx^{\partial(\bfk)} u^{\eta(\bfk)})
&= \sum_{j=1}^r k_j A_j \bfx^{(k_1,\dots,k_{j-1},k_j - 1,k_{j+1},\dots,k_r)} \\
&\quad - \sum_{j=1}^r b_j A_j \bfx^{(b_1,\dots,b_{j-1},b_j - 1,b_{j+1},\dots,b_r)}u^{\eta(\bfk)} - \eta(\bfk) x_i \bfx^{\bfb} u^{\eta(\bfk)-1},
\end{aligned}
$$
where  
$\bfx^{\partial(\bfk)}=\bfx^\bfb=\prod_{j=1}^r x_j^{b_j}.$

Without loss of generality, we assume that the minimum of $\lan c_i, \cdot\ran$ is achieved at $0$ in $P\subset N$, i.e.\ $v_P(c_i)=0\in N$. Then by Lemma \ref{lem bfk} we have $$\eta(\partial(\bfe_j+z\bfe_i)+\bfe_i)=\eta(\bfe_j+(z+1)\bfe_i)-\eta(\bfe_j+z\bfe_i)$$
and thus 
$$n_j=\sum_{z=0}^\infty\left(\eta(\bfe_j+(z+1)\bfe_i)-\eta(\bfe_j+z\bfe_i)\right)=$$
$$=\sum_{z=0}^{\infty}\big(\left(\eta(c_j)+(z+1)\eta(c_i)-\eta(c_j+(z+1)c_i)\right)-\left(\eta(c_j)+z\eta(c_i)-\eta(c_j+zc_i)\right)\big)=\eta(c_j).$$
Thus
$$A_j=\eta(c_j)\bfx^{\partial(\bfe_i+\bfe_j)}u^{\eta(\bfe_i+\bfe_j)-1}.$$ 
Since $D_i(f_\bfk)$ is homogeneous, we see that $D_i(f_\bfk)\in \cI_P$ because  
\begin{equation}\label{eq lhs rhs}
\sum_{j=1}^rk_j\eta(c_j)=\left(\sum_{j=1}^rb_j\eta(c_j)\right)+\eta(\bfk).
\end{equation}
Indeed, the left-hand side (LHS) of \eqref{eq lhs rhs} is the $\ZZ$-coordinate of 
$
\deg(\bfx^\bfk)\in \tM = M\oplus \ZZ,
$
and the right-hand side (RHS) of \eqref{eq lhs rhs} is the $\ZZ$-coordinate of 
$$
\deg(\bfx^{\partial(\bfk)}u^{\eta(\bfk)})\in \tM = M\oplus \ZZ.
$$
\end{proof}

As at the end of the proof above, for any homogeneous polynomial $g(\bfx,u)$, we will denote its degree by $\deg(g(\bfx,u))\in \tM$.

\begin{proposition}\label{prop 52prop}
It holds that $T^1_{\paar}(-r)\cong T^1_{X}(-r)$ for all $r$ except for 
$$
r\in \{R^*-s~|~s\in \partial(S),~(\varphi_s=0)\cap P~\text{is a face of $P$, which is not a vertex}\},
$$
for which it holds that $\dim_\CC T^1_{\paar}(-r)=1+\dim_\CC T^1_X(-r)$.
\end{proposition}
\begin{proof}
    Let $E:=(\varphi_s=0)\cap P$ be a face of $P$ for some $s\in \partial(P)$. 
    Let $s_E$ be the element of the Hilbert basis of $S_P$, such that $s_E^{\perp}\cap \sigma$ equals the face spanned by $\{(v,1)\in \tN\mid v\in E\}$, and let $x_E$ be the corresponding variable (if $E=P$ we have  $s_E=0$ and $x_E=1$). We define
    $$l_E:=\min\{l\in \NN \mid \bfx^\bfk-x_E^{n}u^{l}\in \cI_P, \text{ where $\bfk\in \NN^r$, $n\in \NN$}\}.$$
    Clearly $l_E\geq 1$ and there exists $f(\bfx,u):=\bfx^\bfk-x_E^{n_E}u^{l_E}\in \cI_P$ for some $n_E\in \NN$, $\bfk\in \NN^r$. 
    Let us show that there does not exist a derivation  
$$
D = \sum_{j=1}^r A_j \frac{\partial}{\partial x_j} + x_E \frac{\partial}{\partial u}
$$
such that  
$
D(f(\bfx,u)) \in \cI_P.
$
Indeed, we have  
$$
D(\bfx^\bfk) - D(x^{n_E}_E) u^{l_E} - l_E x^{n_E+1}_E u^{l_E-1} \not\in \cI_P,
$$
by definition of $l_E$.

Thus, the one-dimensional vector space $N_{\partial X|X}(-(R^*-s_E))$ (generated by the element $u \mapsto s_E$, since $N_{\partial X|X} = \Hom_{A'}((u)/(u)^2, A')$) does not lie in the image of the map  
$$
T_{X|\partial X} \xrightarrow{\varphi} N_{\partial X|X}.
$$
Therefore, we obtain  
$$
\dim_\mathbb{C} T^1_{\paar}(-(R^*-s_E)) = 1 + \dim_\mathbb{C} T^1_X(-(R^*-s_E)).
$$
Finally, by Lemma \ref{lem 51}, we see that for any other $m \in \tM$ (i.e., $m \neq R^*$ and $m \neq R^*-s_E$ for some face $E\subset P$ that is not a vertex), we have  
$
\dim_\mathbb{C} T^1_{\paar}(-m) = \dim_\mathbb{C} T^1_X(-m).
$
\end{proof}

\begin{definition}\label{def 61}
  Let $P$ be a polygon and let $m\in \tM$ be such that $m\neq kR^*$ for any $k\in \ZZ$. We say that $P$ is \emph{$m$-mutable} if there exists a line segment $Q\subset (\pi_M(m)=0)$ of lattice length $1$ such that $(m,Q)$ is a deformation pair of $P$. %For every $m\in \tM$, we fix such a segment $Q$ and denote the mutated polytope as $P_m:=P_{(m,Q)}$.
\end{definition}

\begin{remark}
    Equivalently, $P$ is $m$-mutable if and only if there exists a Laurent polynomial $f$ with $\newt(f) = P$ that is $m$-mutable.

\end{remark}

For every edge $E=[v,w]$, let $s_{E}$ be the fundamental generators of the dual cone chosen such that $s_{E}^{\perp}\cap \sigma$ equals the face spanned by $(v,1)$ and $(w,1)$. We denote 
\begin{equation}\label{def ce}
c_E:=\pi_M(s_E)\in M
\end{equation}
to be the projection of $s_E$ to $M$, i.e.\ $s_E=(c_E,\eta_P(c_E))$. We call $c_E$ also the \emph{normal vector} of an edge $E$.
The next proposition connects dimension of $T^1_{(X_P,\partial X_P)}$ with mutations of $P$.

\begin{proposition}\label{t1 prop}
Let $P$ be a polygon, and let $m \neq k R^*$ for any $k \in \NN$. Then
    $$
    \dim_{\CC}T^1_{\paar}(-m) =
    \begin{cases} 
    1 & \text{if $P$ is $m$-mutable and } \max_{v \in P} \varphi_m(v) \geq 1, \\ 
    0 & \text{otherwise}.
    \end{cases}
    $$
For any $k\in \NN$ let $l_k$ denote the number of edges of $P$ that have lattice length greater or equal to $k$. It holds that 
$$
\dim_{\CC}T^1_{\paar}(-kR^*)=l_k-2.
$$
\end{proposition}

\begin{proof}
By \cite[Theorem 4.4]{Alt00} it follows that 
    \begin{equation}\label{eq t11}
    \dim_{\CC}T^1_{X}(-m) =
    \begin{cases} 
    1 & \text{if $P$ is $m$-mutable and } \max_{v \in P} \varphi_m(v) \geq 2, \\ 
    0 & \text{otherwise},
    \end{cases}
    \end{equation}
    if $m\neq k R^*$ for any $k\in \NN$, and 
        \begin{equation}\label{eq t12}
    \dim_{\CC}T^1_{X}(-kR^*) =
    \begin{cases} 
    l_k-3 & \text{if $k=1$} \\ 
    l_k-2 & \text{if $k\geq 2$}.
    \end{cases}
    \end{equation}
By Proposition \ref{prop 52prop} we conclude the proof.
\end{proof}

\begin{example}\label{ex 53}
    Let $P=\conv\{(0,0),(4,0),(0,5)\}$ and $X=X_P$. We denote the edges of $P$ by $E=\conv\{(0,0),(4,0)\}$, $F=\conv\{(0,0),(0,5)\}$ and $G=\conv\{(4,0),(0,5)\}$. Let $\mathbb{Z}_{\geq 1}$ denote the set of positive integers. We have $$T^1_{\paar}(-m)\ne 0$$ if and only if $m\in \cM_1\cup \cM_2$, where 
$$
\cM_1 := 
\big\{ nR^* - k s_E, \, nR^* - k s_F \mid n \in \{1,2,3,4\}, \, k \in \mathbb{Z}_{\geq 1} \big\},
$$
$$
\cM_2:=\big\{ 5R^* - k s_F \mid k \in \mathbb{Z}_{\geq 2} \big\}
\cup 
\big\{ R^* - k s_G \mid k \in \mathbb{Z}_{\geq 1} \big\}\cup \{R^*\}.
$$
If $m\in \cM_1\cup \cM_2$ we have $\dim_\CC T^1_{\paar}(-m)=1$. 
 Moreover, $T^1_{X}(-m)\ne 0$ if and only if $$m\in \cM_1\cup \cM_2\setminus \{R^*-s_E,R^*-s_F,R^*-s_G,R^*\}.$$
\end{example}

\section{Mutable deformations of polygons}\label{sec 6}

In this section, we construct multi-parameter deformations of affine Gorenstein toric varieties associated with polygons, using mutations of Laurent polynomials.

Let $X_P$ be a three-dimensional affine Gorenstein toric variety associated with the polygon $P\subset N_{\RR}$. 
Recall Definition \ref{def 61}. 
We define
$$
\cE(P):=\left\{ m\in \tM \mid P \text{ is } m\text{-mutable}\right\}.
$$
Additionally, for a Laurent polynomial $f\in \CC[N]\cong \CC[\ZZ^2]$, we denote  
\begin{equation}\label{eq cmf}
\cM(f) := \left\{ m \in \tM \mid f \text{ is } m\text{-mutable} \right\}\subset \cE(\newt(f)).
\end{equation}
For every $m\in \cE(P)$, we fix a line segment $Q$ of lattice length $1$, such that $(m,Q)$ is a deformation pair of $P$. We denote $P_m:=P_{(m,Q)}$. 
For $m,r\in \cE(P)$, we define
\begin{equation}\label{psimr}
\psi_{m}(r) := 
\begin{cases}
\psi_{(m,Q)}(r), & \text{if } r\ne m, \\
-m, & \text{if } r=m,
\end{cases}
\end{equation}
where recall the definition of $\psi_{(m,Q)}(r)$ from \eqref{eq psimr}.
\begin{lemma}\label{lem 61}
    The map $\psi_m:\cE(P)\to \cE(P_m)$ is a bijection. Moreover, if $f$ is a Laurent polynomial with $\newt(f)=P$ that is $m$-mutable with $\newt(\mut_m^gf)=P_m$, then $\psi_m:\cM(f)\to \cM(\mut_m^gf)$ is also a bijection.
\end{lemma}
\begin{proof}
    This follows immediately by definition and Proposition \ref{prop mutability}: the inverse of this map is $\psi_{-m}$. 
\end{proof}

For $m\in \cE(P)$, we define 
   \begin{equation}\label{muttm}
\mut_{m}:\tM\to \tM,~~~   \mut_{m}(s):=\xi_{(m,Q)}(s)\in \tM,
   \end{equation}
where recall the definition of $\xi_{(m,Q)}$ from \eqref{eq ximq}.
   \renewcommand{\tr}{\tilde{r}}
   
From now on, let $\{s_1, \dots, s_{\tr}, R^*\}$ denote a generating set of the monoid $S_P$. This set is not necessarily the Hilbert basis, which is the set of minimal generators.

Assume also that $\mut_m(s_1), \dots, \mut_m(s_{\tr}), R^*$ form a generating set of the monoid $S_{P_m}$, which we may assume by Lemma \ref{st lem}.
 We write $\bfy^{\bfk}:=\prod_{j=1}^{\tr} y_j^{k_j}$, which is a monomial of $\tM$-degree $\deg(\bfy^\bfk)=\sum_{j=1}^{\tr} k_j \mut_{m}(s_j)$.

Let 
$$
\bft_P:=\{t_m\mid m\in \cE(P)\}
$$
be the set of deformation parameters, where each $t_m = t_{(m,Q)}$ corresponds to a deformation pair $(m,Q)$ of $P$. 
Let $\cM\subset \cE(P)$ and $\bft_\cM:=\{t_r\mid r\in \cM\}$. 
\begin{definition}\label{mmutable def}
    Let $m \in \cM$. We say that $F_\bfk(\bfx,u,\bft_\cM)$ is \emph{$t_m$-mutable} if, after making the substitutions:
    \begin{itemize}
        \item $x_i$ is replaced by $y_i$,
        \item $t_r$ is replaced by $t_{\psi_m(r)}$ for $r \in \cE(P)$, $r \neq m$,
        \item $t_m$ is set to $1$,
    \end{itemize}
    the resulting expression $F_\bfk(\bfy,u,t_{\psi_m(r)} \mid r\in \cM, r\ne m)$ can be homogenized by $t_{-m}$ in such a way that every monomial has $\tM$-degree $\sum_{j=1}^{\tr}k_j \mut_{m} s_j$. That is, every monomial can be multiplied by $t_{-m}^i$ for some $i \in \NN$ to achieve the desired degree. 
    If we can homogenize $$F_\bfk(\bfy,u,t_{\psi_m(r)} \mid r\in \cM, r\ne m)$$ by $t_{-m}$, then we denote its homogenization by $$\mut_{t_m} F_\bfk \in \CC[\bfy,u][[t_{\psi_m(r)} \mid r\in \cM]]\subset \CC[\bfy,u][[\bft_{P_m}]].$$
    %where $$\bft_{P_m}=\{t_{\psi_m(r)}\mid r\in \cM\}=\{t_s\mid s\in \cE(P_m)\}.$$
\end{definition}

\begin{lemma}\label{lem 6464}
Let 
 $$\cM_1:=\{r\in \cM\mid \pi_M(r)\text{ and }\pi_M(m)\text{ are colinear in }M\cong\ZZ^2,~r\ne m\},$$
$$\cM_2:=\{r\in \cM\mid \pi_M(r)\text{ and }\pi_M(m)\text{ are not colinear in }M\cong\ZZ^2\}.$$
    $F_\bfk(\bfx,u,\bft_\cM)$ is $t_m$-mutable, if for every monomial 
    \begin{equation}\label{eq eqeq aup}
    a\cdot u^p\cdot t^k_{m}\prod_{r\in \cM_1}t^{p_r}_{r}\prod_{r\in \cM_2}t_r^{n_r}\prod_{j=1}^{\tr} x_j^{b_j}
    \end{equation}
    of $F_{\bfk}(\bfx,u,\bft_\cM)$, where $k,p,p_r,n_r,b_j\in \NN$ and $a\in \CC$, it holds that  
    \begin{equation}\label{eq ksum 2}
    k\leq \sum_{j=1}^{\tr}\eta_{Q}(k_j\pi_M(s_j))+\sum_{r\in \cM_2}\eta_Q(-n_r\pi_M(r))-\sum_{r\in \cM_1}p_r-\sum_{j=1}^{\tr}\eta_Q(b_j\pi_M(s_j)).
    \end{equation}
\end{lemma}
\begin{proof}
This follows immediately by definition. Applying the substitutions in Definition \ref{mmutable def} we see that the difference of $\tM$-degree $\deg(\bfx^\bfk)$ of $\bfx^\bfk$ and $\tM$-degree $\deg(\bfy^\bfk)$ of $\bfy^\bfk$ is equal to 
$$\deg(\bfy^\bfk)-\deg(\bfx^\bfk)=\left(-\sum_{j=1}^{\tr}\eta_Q(k_j\pi_M(s_j))\right)m,$$ 
for every $\bfk\in \NN^{\tr}$ (in particular it holds for $\bfb=(b_1,...,b_{\tr})\in \NN^{\tr}$). 
Moreover, we have
$$\deg\left( \prod_{r\in \cM_2}t^{n_r}_{\psi_m(r)} \right)-\deg\left( \prod_{r\in \cM_2}t^{n_r}_{r}\right)=\left(\sum_{r\in \cM_2}\eta_Q(-n_r\pi_M(r))\right)m$$
and 
$$\deg\left( \prod_{r\in \cM_1}t^{p_r}_{\psi_m(r)} \right)-\deg\left( \prod_{r\in \cM_1}t^{p_r}_{r}\right)=-\left(\sum_{r\in \cM_1}p_r\right)m.$$
From this the claim follows.   
\end{proof}

\begin{example}
We continue with Example \ref{ex 3.6}.
 We have $$F_\bfk(\bfx,u,t_{-m})=x_2x_6-ux_5-t_{-m}x_3,$$
 $$F_{\bfa+\bfk}=x_2x_4x_6^2-ux_5^3-t_{-m}^2x_3u-2t_{-m}x_5u^2,$$
 $$F_{\partial(\bfk)+\bfa}=x_4x_5x_6-x_5^3-t_{-m}x_5u,~~~F_{\bfk_1+\bfa}=x_3x_4x_6-u^2x_5-t_{-m}x_3u.$$

    We are going to show that $F_\bfk$, $F_{\bfa+\bfk}$, $F_{\partial(\bfk)+\bfa}$ and $F_{\bfk_1+\bfa}$
 are $t_{-m}$-mutable: replacing $x_i$ by $y_i$, with lattice degree $\deg y_i=z_i$ from Example \ref{ex3}, and setting $t_{-m}$ to $1$, we get from $F_{\bfk}$ the term $y_2y_6-uy_5-y_3.$ Homogenizing by $t_m$ gives us 
 $$\mut_{t_{-m}}F_\bfk=y_2y_6-t_muy_5-y_3=y_2y_6-y_3-t_muy_5,$$
 since $\deg(y_2)=(-1,1,0)$, $\deg(y_6)=(1,0,0)$, $\deg(y_5)=(0,-1,2)$, $\deg(y_3)=(0,1,0)$ and $\deg(t_m)=m=(0,2,-3)$. 
   In the same way, we see that  $F_{\bfa+\bfk}$, $F_{\partial(\bfk)+\bfa}$ and $F_{\bfk_1+\bfa}$ are $t_{-m}$ mutable, which gives us a relation 
   \begin{equation}\label{eq murel}
   \mut_{t_{-m}}F_{\bfa+\bfk}-\bfy^\bfa \mut_{t_{-m}}F_{\bfk}-u\mut_{t_{-m}}F_{\partial(\bfk)+\bfa}-\mut_{t_{-m}}F_{\bfk_1+\bfa}=0,
   \end{equation}
since $$\mut_{t_{-m}}F_{\bfa+\bfk}=y_2y_4y_6^2-t_m^2uy_5^3-y_3u-2t_my_5u^2,$$ 
$$\mut_{t_{-m}}F_{\partial(\bfk)+\bfa}=y_4y_5y_6-t_my_5^3-y_5u,$$
$$\mut_{t_{-m}}F_{\bfk_1+\bfa}=y_3y_4y_6-t_muy_5-y_3u.$$
Note that \eqref{eq murel} lifts a relation $f_{\bfa+\bfk}(\bfy,u)-\bfy^\bfa f_{\bfk}(\bfy,u)-f_{\partial_{P_{m}}(\bfk)+\bfa}(\bfy,u)$, where $f_{\partial_{P_{m}}(\bfk)+\bfa}(\bfy,u)=y_3y_4y_6-y_3u$, which is not a coincidence as we will see in \eqref{eq 68}. 
\end{example}

\begin{definition}\label{def unobs}
    We say that $(X_P,\partial X_P)$ is \emph{unobstructed} in 
    $$\bft_{\cM}=\left\{t_m\mid m\in \cM\subset \cE(P)\right\}$$
    if there exists a formal deformation 
     \begin{equation}\label{cxnf}
    \{F_\bfk(\bfx,u,\bft)\mid \bfk\in \NN^{\tr}\}
    \end{equation}
    of $(X_P,\partial X_P)$ over $R=\CC[[\bft_\cM]]$ such that the image of the Kodaira--Spencer map is isomorphic to $\bigoplus_{m\in \cM}T^1_{(X_P,\partial X_P)}(-m)$ (see e.g.\ \cite[Section 10]{JP00} for a definition of the Kodaira--Spencer map).   
\end{definition}

We will show that $(X_{P},\partial X_P)$ is unobstructed in $\bft_\cM$ if and only if $(X_{P_m},\partial X_{P_m})$ is unobstructed in $\{t_{\psi_m(r)}\mid r\in \cM\}\subset \bft_{P_m}$, where $m\in \cM$.

Since $\{s_1, \dots, s_{\tr}, R^*\}$ does not necessarily form a Hilbert basis, we can, for every $s \in \tM$, choose $\bfb = (b_1, \dots, b_{\tr}) \in \NN^{\tr}$ such that
%$$
%\partial_P(s) = \sum_{j=1}^{\tr} b_j s_j \in \tM,
%$$
%and such that
\begin{equation}\label{eq mutmut}
\partial_P(s) = \sum_{j=1}^{\tr} b_j s_j \in \tM\quad\text{and} \quad\partial_{P_m}(\mut_m s) = \sum_{j=1}^{\tr} b_j \mut_m s_j.
\end{equation}

\begin{theorem}\label{th main 1}
Assume that $(X_P,\partial X_P)$ is unobstructed in $\bft_{\cM}=\{t_r\mid r\in \cM\subset \cE(P)\}$ and that $F_\bfk\in \CC[\bfx,u][[\bft_\cM]]$ from \eqref{cxnf} is $t_m$-mutable for some $m\in \cM$ and every $\bfk\in \NN^{\tr}$. Moreover, assume that the restriction of $F_\bfk(\bfx,u,\bft_\cM)$ to $F_\bfk(\bfx,u,0,\dots,0,t_r,0,\dots,0)$ coincides with \eqref{eq Fk}, for every $r\in \cM$ and $\bfk\in \NN^{\tr}$.
Then, the pair $(X_{P_m},\partial X_{P_m})$ is unobstructed in $\{t_{\psi_m(r)}\mid r\in \cM\}$.
\end{theorem}

\begin{proof}
Restricting $F_\bfk$ to $t_m\in \bft_\cM$, we obtain, under our assumption, that
$$
F_{\bfk}(\bfx,u,0,\dots,0,t_m,0,\dots,0) = \bfx^\bfk - \sum_{i=0}^{\eta_Q(\bfk)} {\eta_Q(\bfk) \choose i} t_m^i \bfx^{\partial_P(s_\bfk - im)} u^{n_P(s_\bfk - im)}.
$$
By \eqref{eq mutmut} 
we see that
\begin{multline} \label{eq 68}
\mut_{t_m} \big( F_{\bfk}(\bfx,u,0,\dots,0,t_m,0,\dots,0) \big)=\\
=\bfy^{\bfk} - \sum_{i=0}^{\eta_Q(\bfk)} {\eta_Q(\bfk) \choose i} t_{-m}^{\eta_Q(\bfk)-i} \bfy^{\partial_P(s_\bfk - im)} u^{n_P(s_\bfk - im)}.
\end{multline}
Note that \eqref{eq 68} provides a lift of
$$
\tilde{f}_\bfk := \bfy^\bfk - \bfy^{\partial_{P_m}(\bfk)} u^{\eta_{P_m}(\bfk)},
$$
since for $i = \eta_Q(\bfk)$ in the above sum, we obtain
$$
\bfy^{\partial_P(s_\bfk - \eta_Q(\bfk)m)} u^{n_P(s_\bfk - \eta_Q(\bfk)m)}
= \bfy^{\partial_{P_m}(\bfk)} u^{\eta_{P_m}(\bfk)}.
$$

Moreover, since we have a formal deformation over $\CC[[\bft_\cM]]$, we know that there exists $$o_{\bfa,\bfk}(\bfx,u,\bft_\cM)\in \CC[\bfx,u][[\bft_\cM]],$$ such that 
$$
R_{\bfa,\bfk}:=F_{\bfa+\bfk}-\bfx^\bfa F_{\bfk}-o_{\bfa,\bfk}=0
$$
and that $R_{\bfa,\bfk}$ are lifts of $r_{\bfa,\bfk}$ (see Proposition \ref{acf1}, where $r_{\bfa,\bfk}$ was introduced). 
Thus,
\begin{equation}\label{mut fmut}
\mut_{t_m}F_{\bfa+\bfk}=\bfy^\bfa \mut_{t_m}F_{\bfk}+\tilde{o}_{\bfa,\bfk},
\end{equation}
for some $\tilde{o}_{\bfa,\bfk} \in \CC[\bfx,u][[t_{\psi_m(r)}\mid r\in \cM]]$, since $F_{\bfa+\bfk}$ is $t_m$-mutable. By \eqref{eq 68}, we see that 
$$
\widetilde{R}_{\bfa,\bfk}:=\mut_{t_m}F_{\bfa+\bfk}-\bfy^\bfa \mut_{t_m}F_{\bfk}+\tilde{o}_{\bfa,\bfk}
$$
is a lift of 
$$
\tilde{r}_{\bfa,\bfk}:=\tilde{f}_{\bfa+\bfk}-\bfx^\bfa \tilde{f}_{\bfk}-u^{\eta_{P_m}(\bfk)}\tilde{f}_{\bfa+\partial(\bfk)}.
$$

Thus, $\{\mut_{t_m}F_{\bfk}\mid \bfk\in \NN^{\tr}\}$ is a formal deformation of $(X_{P_m},\partial X_{P_m})$ over $\CC[[t_{\psi_m(r)}\mid r\in \cM]]$. We now show that the image of its Kodaira--Spencer map is isomorphic to
$$
\bigoplus_{r\in \cM} T^1_{(X_{P_m},\partial X_{P_m})}(-\psi_m(r)).
$$
Indeed, for each $r\in \cM$, the Kodaira--Spencer class of the one-parameter deformation
$$
F_{\bfk}(\bfx, u, 0, \dots, 0, t_{\psi_m(r)}, 0, \dots, 0)
$$
 of $X_{P_m}$, spans  $T^1_{(X_{P_m},\partial X_{P_m})}(-\psi_m(r))\subset T^1_{(X_{P_m},\partial X_{P_m})}$. 
 Let $\bfk$ be such that for all $i \in \{1, \dots, \tr\}$ with $k_i c_i \ne 0$, the function $c_i$  achieves its minimal value at the same vertex of the line segment $$Q\subset (\pi_M(m)=0) \text{, such that $(m,Q)$ is a deformation pair of $P$}.$$ Then
$$
\mut_{t_m}\left(F_{\bfk}(\bfx, u, 0, \dots, 0, t_r, 0, \dots, 0)\right) = F_\bfk(\bfy, u, 0, \dots, 0, t_{\psi_m(r)}, 0, \dots, 0).
$$
Note that since \eqref{eq mutmut} holds, it follows that there exists at least one such $\bfk$ for which the deformation parameter $t_r$ appears in the one-parameter deformation $F_{\bfk}(\bfx, u, 0, \dots, 0, t_r, 0, \dots, 0)$. This completes the proof of the claim.
\end{proof}

In the following theorem we will need the following constructions.

\begin{construction}\label{cons 1}
Assume there exists a formal deformation
$$
\left\{ F_{\bfk}(\bfx, u, \bft_{\cS}) \mid \bfk \in \mathbb{N}^{\tr} \right\}
$$
of $(X_P, \partial X_P)$ over $\mathbb{C}[[\bft_{\cS}]]$, where $\bft_{\cS}:=\{t_r\mid r\in \cS\subset \cE(P)\}$. We are going to construct a formal deformation $\{\widetilde{F}_\bfb(\bfx, u, \bft_{\cS}) \mid \bfb \in \NN^{\tr}\}$; see \eqref{eq fb} for the definition of $\widetilde{F}_\bfb$, and \eqref{eq rcrt} for the description of the corresponding linear relations.

There exist elements $o_{\bfa, \bfk}(\bfx, u, \bft_{\cS})$ satisfying
\[
o_{\bfa, \bfk}(\bfx, u, 0) = u^{\eta_P(\bfk)} F_{\bfa + \partial(\bfk)}(\bfx, u,0)=u^{\eta_P(\bfk)} f_{\bfa + \partial(\bfk)}(\bfx, u),
\]
and such that the
\[
R_{\bfa, \bfk} := F_{\bfa + \bfk} - \bfx^\bfa F_{\bfk} - o_{\bfa, \bfk}
\]
is a linear relation.
Consider any term
$$
\tilde{a} \bfx^{\tilde{\bfi}}u^{\tilde{p}} \bft_{\cS}^{\tilde{\bfj}} F_{\bfb}
$$
appearing in $R_{\bfa, \bfk}$, where $\tilde{a} \in \mathbb{C}\setminus \{0\}$, $\tilde{p}\in \NN$, $\tilde{\bfi}, \bfb \in \mathbb{N}^{\tr}$, and $\tilde{\bfj} \in \mathbb{N}^{|\cS|}$. Within $F_{\bfb}$, consider any term
\begin{equation}\label{term 1 eq revised}
a \bfx^\bfi u^p\bft^\bfj_\cS,
\end{equation}
where $a \in \mathbb{C}\setminus \{0\}$, $p\in \NN$, $\bfi \in \mathbb{N}^{\tr}$, and $\bfj \in \mathbb{N}^{|\cS|}$. Assume $\bfj \neq 0$ (that is, at least one deformation parameter appears) and that $F_\bfi\ne 0$. We replace \eqref{term 1 eq revised} by
$$
a (-F_{\bfi} + \bfx^{\bfi}) u^p\bft^\bfj_\cS,
$$
and add the term $a \tilde{a} \bfx^{\tilde{\bfi}} u^{p+\tilde{p}} \bft_{\cS}^{\bfj+\tilde{\bfj}} F_{\bfi}$ to $R_{\bfa, \bfk}$ to ensure that, even after these replacements, we still have a linear relation.

Let $m_R$ denote the maximal ideal of $R=\mathbb{C}[[\bft_{\cS}]]$. Repeating this procedure, we construct (for any $n \in \mathbb{N}$) linear relations
$$
\widetilde{R}_{\bfa, \bfk} := \widetilde{F}_{\bfa + \bfk} - \bfx^\bfa \widetilde{F}_{\bfk} - \tilde{o}_{\bfa, \bfk},$$
where each $\widetilde{F}_{\bfb}$ appearing in $\widetilde{R}_{\bfa, \bfk}$, modulo $m_R^n$, contains only terms of the form $a'u^{p'} \bfx^{\bfi'} \bft^{\bfj'}$, where $a' \in \mathbb{C}\setminus \{0\}$, $p'\in \NN$, $\bfi' \in \mathbb{N}^{\tr}$, $\bfj' \in \mathbb{N}^{|\cS|}$, and $F_{\bfi'}=0$.
 Moreover,
\[
\tilde{o}_{\bfa, \bfk}(\bfx, u, 0) = u^{\eta_P(\bfk)} \widetilde{F}_{\bfa + \partial(\bfk)}(\bfx, u, 0)=u^{\eta_P(\bfk)}f_{\bfa + \partial(\bfk)}(\bfx, u),
\]
and 
$\tilde{o}_{\bfa, \bfk}(\bfx, u, \bft_{\cS})$ contains terms of the form
\begin{equation}\label{eq upam}
a \bfx^\bfi u^p \bft_{\cS}^\bfj \widetilde{F}_{\bfb}(\bfx, u, \bft_{\cS}),
\end{equation}
where $a \in \mathbb{C}\setminus \{0\}$, $p\in \NN$, $\bfi, \bfb \in \mathbb{N}^{\tr}$, and $\bfj \in \mathbb{N}^{|\cS|}$. We replace $\bfx^\bfi \widetilde{F}_{\bfb}(\bfx, u, \bft_{\cS})$ by $\widetilde{F}_{\bfi + \bfb}(\bfx, u, \bft_{\cS}) - \tilde{o}_{\bfi, \bfb}(\bfx, u, \bft_{\cS})$, whenever $\bfi \neq 0$, noting that these two expressions are indeed equal.

Repeating this procedure, we obtain, for every $n \in \mathbb{N}$, a linear relation
\[
R'_{\bfa, \bfk}(\bfx, u, \bft_{\cS}) := \widetilde{F}_{\bfa + \bfk}(\bfx, u, \bft_{\cS}) - \bfx^\bfa \widetilde{F}_{\bfk}(\bfx, u, \bft_{\cS}) - o'_{\bfa, \bfk}(\bfx, u, \bft_{\cS}),
\]
where $o'_{\bfa, \bfk}(\bfx, u, \bft_{\cS})$ contains, modulo $m_R^n$, terms of the form
\begin{equation}\label{eq upam 2}
b u^k\bft_{\cS}^\bfn \widetilde{F}_{\bfb}(\bfx, u, \bft_{\cS}),
\end{equation}
with $b \in \mathbb{C}$, $k\in \NN$, $\bfb\in \mathbb{N}^{\tr}$ and $ \bfn \in \NN^{|\cS|}$, and where $\widetilde{F}_{\bfb}(\bfx, u, \bft_{\cS})$ satisfies the same conditions as before.
In particular, modulo $m_R^n$, for each $\bfb \in \mathbb{N}^{\tr}$, we have
\begin{equation}\label{eq fb}
\widetilde{F}_{\bfb}(\bfx, u, \bft_{\cS}) = \bfx^\bfb - \sum_{\bfj \in J_{\bfb} \subset \mathbb{N}^{|\cS|}} n_{\bfj}(\bfb) \bft_{\cS}^{\bfj} \chi^{s_{\bfb} - \deg(\bft_{\cS}^{\bfj})},
\end{equation}
where
$$
\bft_{\cS}^{\bfj} := \prod_{r \in \cS} t_r^{j_r},
$$
and $n_{\bfj}(\bfb) \in \mathbb{C}$ with $n_0(\bfb) = 1$, so that $\widetilde{F}_{\bfb}(\bfx, u, 0) = f_{\bfb}(\bfx, u)$. 
Additionally, modulo $m_R^n$,  we have
\begin{equation}\label{eq rcrt}
R'_{\bfa, \bfk}(\bfx, u, \bft_{\cS}) = \widetilde{F}_{\bfa + \bfk}(\bfx, u, \bft_{\cS}) - \bfx^\bfa \widetilde{F}_{\bfk}(\bfx, u, \bft_{\cS}) - o'_{\bfa, \bfk}(\bfx, u, \bft_{\cS}),
\end{equation}
where
$$
o'_{\bfa, \bfk}(\bfx, u, \bft_{\cS}) = \sum_{\bfj \in J_{\bfk}} n_{\bfj}(\bfk) \bft_{\cS}^\bfj u^{n_P(s_\bfk-\deg(\bft^\bfj_\cS))}\widetilde{F}_{\bfa + \partial(s_{\bfk} - \deg(\bft_{\cS}^\bfj))}(\bfx, u, \bft_{\cS}),
$$
with $\partial(s_{\bfk} - \deg(\bft_{\cS}^\bfj)) \in \mathbb{N}^{\tr}$. 
In particular, we have $o'_{\bfa, \bfk}(\bfx,u, 0) = f_{\bfa + \partial(\bfk)}(\bfx,u)$, which implies that $R'_{\bfa, \bfk}(\bfx, u, 0) = r_{\bfa, \bfk}(\bfx,u)$.
\end{construction}

\begin{construction}\label{cons 2}
Let $\{s_1, \dots, s_{\tr},R^*\}$ be a generating set of $S_P$ satisfying \eqref{eq mutmut} and let $G\subset P$ be the lattice point $0\in N$. The generating set of $S_G$ is $$\{\tilde{s}_1, \dots, \tilde{s}_{\tr},R^*\},$$
where $\tilde{s}_i=(\pi_M(s_i),0)$. Let $z_i$ be the corresponding variables. 
%We define a variable $z_i$ with $\deg(z_i)=c_i\in M$, where $s_i=(c_i,\eta_P(c_i))$. 
We are going to deform 
$X_G:=\spec \CC[\bfz]/(f_{\bfk}(\bfz)\mid \bfk \in \NN^{\tilde{r}})$, where $$f_\bfk(\bfz):=\bfz^\bfk-\bfz^{\partial_P(\bfk)},$$
with the set of deformation parameters $\bft_\cM:=\{t_r\mid r\in \cM\subset M\}$. For every $r\in M$ let  $Q_r\subset (r=0)$ be a line segment of lattice length $1$.
%the lattice point $0\in N$, and let $\cM \subset \cE(P)$ be such that the affine function $\varphi_m$ takes value $0$ at $P$ for all $m\in \cM$. Let ${s_1, \dots, s_{\tr}}$ be a generating set of $S_P$. Note that each $s_i$ has last coordinate equal to $0$, and that $R^*$ to be part of the generating set. Also note that $\Spec \CC[S_P]$ is an algebraic torus.
%Let $r\in \cM\subset \cE(P)\subset \tM$, and write $t_r = t_{(r, Q_r)}$ for the corresponding deformation parameter, where $Q_r\subset (\pi_M(r)=0)$ is a line segment of lattice length $1$. 

Fix $m \in \cM$ and write $Q := Q_m$. We define
$$\cM_1:=\{r\in \cM\mid \text{$r$ and $m$ are colinear}\},$$
$$\cM_2:=\{r\in \cM\mid \text{$r$ and $m$ are not colinear}\}.$$ 
By Lemma \ref{lem bfk} we easily see that the linear relations among
    \begin{equation}\label{eq fbfk2}
    F_\bfk(\bfz,\bft_\cM):=\bfz^\bfk-\bfz^{\partial_P(\bfk)}\left(1+\sum_{r\in \cM_1}\chi^{-r}t_{r}\right)^{\eta_{Q}(\bfk)}\prod_{r\in \cM_2}(1+\chi^{-r}t_{r})^{\eta_{Q_r}(\bfk)}
    \end{equation}
    are
    $$R_{\bfa,\bfk}(\bfz,\bft_\cM):=F_{\bfa+\bfk}-\bfx^\bfa F_\bfk-\left(1+\sum_{r\in \cM_1}\chi^{-r}t_{r}\right)^{\eta_{Q}(\bfk)}\prod_{r\in \cM_2}(1+\chi^{-r}t_{r})^{\eta_{Q_r}(\bfk)}F_{\bfa+\partial_P(\bfk)}.$$
  
Thus, we have constructed a deformation $\{F_\bfk(\bfz, \bft_\cM)\mid \bfk \in \NN^{\tr}\}$ of $(X_G,\partial X_G)$. Note that the polytopes $Q_r$ are translation equivalent to $Q$ for all $r \in \cM_1$. By applying Construction~\ref{cons 1}, we obtain a formal deformation of $(X_G,\partial X_G)$ which is $t_m$-mutable (by Lemma \ref{lem 6464}) and such that the restriction to $\CC[[t_r]]$ induces the one-parameter deformation from Corollary~\ref{cor 3.3}, for all $r \in \cM$.
\end{construction}

\begin{example}
Let $$\{(-1,0,0),(0,1,0),(1,0,0),(0,-1,0),(1,-1,0),(1,-2,0),R^*\}$$ be a generating set of $S_G$, where $G:=0\in N\cong \ZZ^2$.
  Let $\chi^{(i,j)}$ correspond to an element $(i,j,0)$ in the generating set and let $\chi^{(0,0)}=1$. Let $t$ and $s$ be deformation parameters with $\deg(t)=(1,0,0)$ and $\deg(s)=(0,-1,0)$. 
    Let $$f_{(-1,0),(1,-2)}=\chi^{(-1,0)}\chi^{(1,-2)}-\left(\chi^{(0,-1)}\right)^2,$$
    $$f_{(0,1),(0,-1),(0,-1)}=\chi^{(0,1)}\left(\chi^{(0,-1)}\right)^2-\chi^{(0,-1)},$$
    $$f_{(-1,0),(0,1),(1,-2)}=\chi^{(-1,0)}\chi^{(0,1)}\chi^{(1,-2)}-\chi^{(0,-1)},$$
    $$f_{(-1,0),(1,-1)}=\chi^{(-1,0)}\chi^{(1,-1)}-\chi^{(0,-1)},$$
    $$f_{(0,1),(0,-1)}=\chi^{(0,1)}\chi^{(0,-1)}-\chi^{(0,0)},~~~f_{(-1,0),(0,-1)}=0.$$
    The deformation described in \eqref{eq fbfk2} is in this case equal to
    $$F_{(-1,0),(1,-2)}=f_{(-1,0),(1,-2)}-s\chi^{(0,1)}\left(\chi^{(0,-1)}\right)^2$$
     and since 
     $$
     F_{(0,1),(0,-1),(0,-1)}=f_{(0,1),(0,-1),(0,-1)}-t\chi^{(-1,0)}\chi^{(0,-1)},
     $$
     we see that the deformation described in \eqref{eq fb} is in our case equal to 
     \begin{equation}\label{eq wtilf-1}
     \widetilde{F}_{(-1,0),(1,-2)}=f_{(-1,0),(1,-2)}-s\chi^{(0,-1)}-st \chi^{(-1,0)}\chi^{(0,-1)}.
     \end{equation}
Similarly, we can see that 
$$F_{(-1,0),(0,1),(1,-2)}=$$
$$=f_{(-1,0),(0,1),(1,-2)}-t\chi^{(-1,0)}\chi^{(0,-1)}-s\chi^{(0,1)}\chi^{(0,-1)}-st\chi^{(-1,0)}\chi^{(0,1)}\chi^{(0,-1)},$$

and thus 
\begin{equation}\label{eq wtilf-2}
\widetilde{F}_{(-1,0),(0,1),(1,-2)}=f_{(-1,0),(0,1),(1,-2)}-t\chi^{(-1,0)}\chi^{(0,-1)}-s-2ts\chi^{(-1,0)}-t^2s\left(\chi^{(-1,0)}\right)^2.
\end{equation}
    We can lift the relation 
     $$f_{(-1,0),(0,1),(1,-2)}-\chi^{(0,1)}f_{(-1,0),(1,-2)}-f_{(0,1),(0,-1),(0,-1)}$$
     by 
     $$\widetilde{F}_{(-1,0),(0,1),(1,-2)}-\chi^{(0,1)}\widetilde{F}_{(-1,0),(1,-2)}-\widetilde{F}_{(0,1),(0,-1),(0,-1)}-s\widetilde{F}_{(0,1),(0,-1)}-st\widetilde{F}_{(0,1),(-1,0),(0,-1)},$$
     which is a relation described in \eqref{eq rcrt}.
     This is indeed a linear relation by \eqref{eq wtilf-1}, \eqref{eq wtilf-2}, and since 
     $$\widetilde{F}_{(0,1),(0,-1)}=f_{(0,1),(0,-1)}-t\chi^{(-1,0)},$$
     and
     $$\widetilde{F}_{(0,1),(0,-1),(0,-1)}=f_{(0,1),(0,-1),(0,-1)}-t\chi^{(0,-1)}\chi^{(-1,0)}.$$
\end{example}

\begin{theorem}\label{th main 2}
    If $X_{P}$ is unobstructed in $\bft_\cM=\{t_r~|~r\in \cM\subset \cE(P)\}$, then, for every $m\in \cM$, there exists a formal deformation 
    $\{F_\bfk(\bfx,u,\bft_\cM)~|~\bfk\in \NN^{\tr}\}$ of $(X_P,\partial X_P)$ over $\CC[[\bft_\cM]]$, such that $F_\bfk(\bfx,u,\bft_\cM)$ is $t_m$-mutable and that the restriction of $F_\bfk(\bfx,u,\bft_\cM)$ to $F_\bfk(\bfx,u,0,\dots,0,t_r,0,\dots,0)$ coincides with \eqref{eq Fk}, for every $r\in \cM$ and $\bfk\in \NN^{\tr}$.
\end{theorem}

\begin{proof}
    Since $X_{P}$ is unobstructed in $\bft_\cM=\{t_m~|~m\in \cM\subset \cE(P)\}$, we have a formal deformation 
    $\{F_{\bfk}(\bfx,u,\bft_\cM)\mid \bfk\in \NN^{\tr}\}$ of $(X_P,\partial X_P)$ over $\CC[[\bft_\cM]]$, such that the restriction of $F_\bfk$ to a single deformation parameter $t_m$ (by setting all others to $0$) coincides with \eqref{eq Fk}, for every $m\in \cM$ and $\bfk\in \NN^{\tr}$. This argument is standard in formal deformation theory, using Corollary~\ref{cor 3.3} and the implicit function theorem for formal power series, since the image of the Kodaira--Spencer map of the one-parameter deformation family from Corollary~\ref{cor 3.3} (for $t=t_m$, $m\in \cM$) equals $T^1_{X_P}(-\deg(t_m)) = T^1_{X_P}(-m)$ (see, e.g., \cite[Section~10]{JP00}).

Without loss of generality, we may assume that $\cM$ consists only of those elements $r \in \cE(P)$ such that $R^* + r \notin \cM$ (at the end, we can replace $t_r$ by a formal expression $t_r + u t_{r - u} + u^2 t_{r - 2u} + \cdots$ in order to obtain the desired deformation). Fix $m \in \cM \subset \cE(P)$, and let $G \subset P$ be the lattice point $0\in N$.

%\red{Let $t_m \in \bft_\cM \subset \bft_P$.  
%We distinguish two cases, depending on whether $\varphi_m$ achieves its maximum value on $P$ along an edge or at a vertex of $P$. 
%Let us begin with the first case: let $t_m \in \bft_\cM$ be such that $\varphi_m$ achieves its maximum value on $P$ along an edge $E$.
% Let $G\subset E\subset P$ be a line segment of lattice length $1$. }

Let 
$$
\{s_1 = (c_1, \eta_P(c_1)), \dots, s_{\tr} = (c_{\tr}, \eta_P(c_{\tr})), R^*\}
$$
be a generating set for $S_P$ that satisfies \eqref{eq mutmut}, and consider the following generating set for $S_G$:
$$
\{\tilde{s}_1 = (c_1, 0), \dots, \tilde{s}_{\tr} = (c_{\tr}, 0),R^*\}.
$$
We will show that the deformation of $(X_P, \partial X_P)$ given by $F_\bfk(\bfx, u, \bft_\cM)$, whose restriction to each deformation parameter coincides with \eqref{eq Fk}, induces a deformation of $(X_G, \partial X_G)$ that is formally isomorphic to the deformation constructed in Construction \ref{cons 2}. Comparing these two deformations will prove the claim.

In $F_{\bfk}(\bfx,u,\bft_\cM)$, we insert $x_j = u^{n_j} z_j$, where $n_j := \eta_P(c_j)$, and denote it by $F_{\bfk}(\bfz,u,\bft_\cM)$. After this insertion, 
$$
F_{\bfk}(\bfx,u,0) = f_{\bfk}(\bfx,u) = \bfx^\bfk -  \bfx^{\partial_P(\bfk)}u^{\eta_P(\bfk)}
$$
becomes
$$
u^{\sum_{j=1}^{\tr} k_j n_j} \tilde{f}_{\bfk}(\bfz) := u^{\sum_{j=1}^{\tr} k_j n_j} \left( \bfz^\bfk - \bfz^{\partial_P(\bfk)}  \right).
$$

Let $n_{\bfk} := \sum_{j=1}^{\tr} k_j n_j$. For each $r \in \cM$, define
$$
T_r := \frac{t_r}{u^{p_r}},
$$
where $p_r \in \mathbb{Z}$ is chosen such that $\deg(T_r)$ takes value $0$ on $G$. This means that $\deg(T_r)=(\pi_M(\deg(t_r)),0)$ Thus, 
    $$
    \widetilde{F}_{\bfk}(\bfz, \bfTT_\cM) := \frac{F_{\bfk}(\bfz, u, \bft_\cM)}{u^{n_{\bfk}}} = \tilde{f}_{\bfk}(\bfz) - o(\bfz, \bfTT_\cM)
    $$
    induces a formal deformation of $(X_G, \partial X_G)$, where $o(\bfz, \bfTT_\cM) \in \mathbb{C}[\bfz][[\bfTT_\cM]]$, and
    $$
    \bfTT_{\cM} := \{T_r \mid r \in \cM\}.
    $$
Indeed, $\widetilde{F}_\bfk(\bfz,\bfTT_\cM)$ is obtained from $F_\bfk(\bfx,u,\bft_\cM)$ by setting $u = 1$, $x_i = z_i$, and $t_r = T_r$, for every $i=1,\dots,\tilde{r}$ and $r\in \cM$.

In Construction \ref{cons 2} we constructed a formal deformation $$\{F'_\bfk(\bfz,T_r) \mid r \in \cM, \bfk\in \NN^{\tr}\}$$ of $(X_G,\partial X_G)$ over $\mathbb{C}[[\bfTT_\cM]]=\mathbb{C}[[T_r \mid r \in \cM]],$ such that the restriction to $\mathbb{C}[[T_r]]$ is the one-parameter deformation family of $(X_G,\partial X_G)$ from Corollary \ref{cor 3.3}.

There exists a formal isomorphism $\Phi$ between this formal deformation and the one given by $$\{\widetilde{F}_\bfk(\bfz, \bfTT_\cM) \mid \bfk \in \NN^{\tr}\},$$ such that $\Phi$ is the identity modulo $(\bfTT_\cM)^2$, where $(\bfTT_\cM)$ denotes the maximal ideal of $\CC[[\bfTT_\cM]]$. This holds because the two formal deformations agree to first order.

The map $\Phi$ acts on coordinates as
$$
\Phi(z_i) = z_i + p_i(\bfz, \bfTT_\cM),
$$
where $p_i \in \CC[\bfz][[\bfTT_\cM]]$, and $p_i = 0 \mod (\bfTT_\cM)^2$ for all $i = 1, \dots, \tr$. %and 
%$$\Phi(F'_\bfk(\bfz,\bfTT_\cM))=\widetilde{F}_\bfk(\bfz,\bfTT_\cM)$$ modulo $$(\Phi(F'_\bfk(\bfz,\bfTT_\cM))\mid \bfk\in \NN^{\tr})=(\widetilde{F}_\bfk(\bfz,\bfTT_\cM)\mid \bfk\in \NN^{\tr}).$$

We observe that $p_i$ cannot contain a monomial of the form
$$
V \;=\; a \cdot T_m^k \cdot \prod_{r \in \cM_1} T_r^{p_r}\prod_{r \in \cM_2} T_r^{n_r} \prod_{j=1}^{\tr} z_j^{b_j},
$$
where $a \in \CC\setminus\{0\}$ and $k,p_r,n_r, b_j \in \NN$, such that the following two conditions are satisfied:
\begin{itemize}
\item[(i)] there does not exist $s \in S_P$ such that $\deg(\tilde{V}) + s = s_i$,
where
$$
\tilde V = a \cdot t_m^k \cdot \prod_{r \in \cM_1} t_r^{p_r} \prod_{r \in \cM_2} t_r^{n_r} \prod_{j=1}^{\tr} x_j^{b_j},
$$
\item[(ii)] 
 $$k>\eta_Q(\pi_M(s_i))+\sum_{r\in \cM_2}\eta_Q(-n_r\pi_M(r))-\sum_{r\in \cM_1}p_r-\sum_{j=1}^{\tr}\eta_Q(b_j\pi_M(s_j)).$$
\end{itemize}

Indeed, suppose that such a monomial $V$ occurs in some $p_i$. Among all pairs $(V,p_i)$ with this property, choose one such that 
\[
k+\sum_{r \in \cM_1} p_r + \sum_{r\in \cM_2} n_r
\]
is minimal. Then one can choose a generating set of $S_P$ together with $\bfk \in \NN^{\tr}$ and $n \in \NN$ such that 
$$
\tilde f_\bfk \;=\; z_i^n z_j \;-\; \bfz^{\partial_P(\bfk)} \;\neq\; 0,
$$
and with the property that the degree $s_j$ (corresponding to $x_j$ in the generating set of $S_P$) takes the value $0$ at $(v,1)\in \tN$, where $v$ is a vertex of $P$ for which
$$
\langle \deg(\tilde V), (v,1)\rangle > \langle s_i, (v,1)\rangle.
$$
Note that such $v\in P$ exists by (i).
In this situation, we can easily show that there exists $n \in \NN$ sufficiently large (and $\bfk \in \NN^{\tr}$) such that 
$$
\Phi(F'_\bfk)\not\in \big(\widetilde{F}_\bfk(\bfz,\bfTT_\cM) \,\big|\, \bfk\in \NN^{\tr}\big).
$$ 
Note that condition (ii) is needed to ensure that the monomial $V^nz_j$ of $\Phi(\tilde{f}_\bfk)$ does not cancel with $\Phi(V')$, where $V'$ is another monomial of $F'_\bfk$ involving some deformation parameters. The largeness of $n$ guarantees the claim, since otherwise, i.e.\ if 
$$
\Phi(F'_\bfk)\in \big(\widetilde{F}_\bfk(\bfz,\bfTT_\cM)\,\big|\, \bfk\in \NN^{\tr}\big),
$$ 
we would obtain a contradiction by lattice degree considerations and by the construction of $\widetilde{F}_\bfa(\bfz,\bfTT_\cM)$ for $\bfa\in \NN^{\tr}$.

Since every $F'_\bfk$ is $T_m$-mutable, it follows that our initial deformation of $$(X_P, \partial X_P)$$ is formally isomorphic to a deformation of $(X_P, \partial X_P)$ that is $t_m$-mutable and whose restriction to the one-parameter deformation over $\CC[[t_r]]$ coincides with the one constructed in Corollary \ref{cor 3.3}, for every $r\in \cM$. 
\end{proof}

As a corollary we get the following theorem.
\begin{theorem}\label{main th 1 un}
    $X_P$ is unobstructed in $\bft_\cM=\{t_m\mid m\in \cM\subset \cE(P)\}$ if and only if  $X_{P_m}$ is unobstructed in $\bft_{\psi_m(\cM)}=\{t_{\psi_m(r)}\mid r\in \cM\}$. Moreover, the general fibre of the unobstructed deformation of $X_P$ over $\CC[[\bft_\cM]]$ is isomorphic to the general fibre of the unobstructed deformation of $X_{P_m}$ over $\CC[[\bft_{\psi_m(\cM)}]]$. 
\end{theorem}
\begin{proof}
The first statement follows from Theorems \ref{th main 1} and \ref{th main 2}. The second statement follows from the construction described in the proofs: we have a family over $\mathbb{P}^1$, with the fibre over $0$ equal to the unobstructed deformation of $X_P$, and the fibre over $\infty$ equal to the unobstructed deformation of $X_{P_m}$. These families are glued over $\mathbb{C}[t_m, t_{-m}] \subset \mathbb{P}^1$, with $t_{-m} = t_m^{-1}$.
\end{proof}

\section{Mutation equivalence classes}\label{sec mut eq classes}
In this section, we conclude the proof of Theorem \ref{thmm} from the introduction by studying mutation equivalence classes of Laurent polynomials and their relationship to unobstructed deformations of affine Gorenstein toric varieties.

Let $P$ be a polygon and $f$ a Laurent polynomial with $\newt(f)=P$. 
For every edge $E$ of $P$ and for positive integers $n, k \in \ZZ_{\geq 1}$, we define  
$$
m_{n,k}^{E} := nR^* - k s_{E} \in \tM,
$$
where $s_{E} = (c_E, \eta_P(c_E))$, as given in \eqref{def ce}.  
For a Laurent polynomial $f$, recall $\cM(f)$ from \eqref{eq cmf}.
For every edge $E$ of $\newt(f)$, let 
$$
n_E:=n_E(f):=\max\{n\in \NN\mid m^{E}_{n,1}\in \cM(f)\} \in \NN,
$$ 
and denote  
\begin{equation}\label{eq me}
    m^E:=m^E(f):=m^E_{n_E,1} \in \cM(f)\subset \cE(\newt(f)).
\end{equation}

We arrange the generators $a^i=(v^i,1)$ for $i=1,\dots,p$ of $\sigma$ in a cycle with $p+1:=1$, such that the vectors  
$d^i=v^{i+1}-v^i$  
orient $P$ counterclockwise. Here, $n=p$, meaning that the number of edges of $P$ is equal to the number of vertices. 
Moreover, for an edge $E=E_i=[v^i,v^{i+1}]$, we denote  
\begin{equation}\label{def ae de}
    a_E:=a_i:=\frac{1}{\ell(E)}\big(a^{i+1}-a^i\big)\in \tN,~~~~
    d_E:=d_i:=\frac{1}{\ell(E)}\big(v^{i+1}-v^i\big) = \pi_N(a_E).
\end{equation}
Recall that $\ell(E)$ denotes  the lattice length of an edge $E$.
For a polytope $P$, we denote by $n(P)$ the sum of the lattice lengths of all edges of $P$:  
\begin{equation}\label{eq np}
    n(P):=\sum_{E;~E \text{ an edge of } P} \ell(E).
\end{equation}
For a Laurent polynomial $f\in \CC[N]\cong \CC[\ZZ^2]$, we denote  
$
n(f) := n(\newt(f)).
$
If $f$ is $m$-mutable, we choose $g$ such that $f$ is $(m,g)$-mutable and $\newt(g) \subset (m = 0)$ is a line segment of lattice length $1$.

Note that since all Laurent polynomials in this paper are normalized, $g$ is completely determined by $\newt(g)$.
We define  
$
\mut_m f := \mut_m^g f.
$ 
Moreover, if $m = m^E$, we write  
$
\mut_E f := \mut_{m^E} f.
$

\begin{definition}
    For two non-parallel edges $E, F$ of $\newt(f)$ with $E \ne F$, we define  
    $$
    \psi_E(F) := \psi_{m^E}(m^F), \text{ cf.\ \eqref{psimr}},
    $$
    and  
    $$
    (\mut_F \circ \mut_E) f := \mut_{\psi_E(F)}(\mut_E f).
    $$
    Similarly, for multiple mutations, we define  
    $$
    (\mut_G \circ \mut_F \circ \mut_E) f := (\mut_{\psi_E(G)} \circ \mut_{\psi_E(F)})(\mut_E f).
    $$
\end{definition}

\newcommand{\m}{\mut}

\begin{definition}
    We say that two Laurent polynomials are \emph{mutation equivalent} if there exists a sequence of mutations mapping $f$ to $g$; more precisely, if there exist $m_i\in \tM$ such that  
    $$
    (\mut_{m_k} \circ \cdots \circ \mut_{m_1}) f = g.
    $$
    Here, $f$ is $m_1$-mutable, $\mut_{m_1} f$ is $m_2$-mutable, and so on. 
\end{definition}

\begin{definition}
We say that a Laurent polynomial $f$ is \emph{minimal} if every Laurent polynomial $h$ that is mutation equivalent to $f$ satisfies $n(h) \geq n(f)$.
\end{definition}

For simplicity, in the following, we write $m(v):=\varphi_m(v)$ for the value of an affine function $\varphi_m$ at $v\in N$. Moreover, we define \emph{the distance} between $a,b\in N$ in the direction of $c\in M$ to be  
$
|\langle c,a \rangle - \langle c,b \rangle|.
$

The following proofs of Proposition \ref{prop 74} and Theorem \ref{th smoothable} introduce a substantial amount of notation, and we refer the reader to Example \ref{ex 76} for further clarification.

\begin{proposition}\label{prop 74}
    Let $f\in \CC[N]\cong \CC[\ZZ^2]$ be a minimal Laurent polynomial with $P = \newt(f)$, and let $E \subset (y = -n_E)$ be a horizontal edge of $P$ such that  
    \begin{equation}\label{eq prop 74}
    n_E = \max\{n_H \mid H \text{ is an edge of } P\}.
    \end{equation}
    Additionally, assume that there is a vertical edge $F \subset (x = n_F)$ of $P$, and in particular, this means that $m^E(0,0) = m^F(0,0) = 0$. Assume that $(0,-n_E)\in E$ and also that among the points of $P$ with maximal $y$-coordinate, there exists at least one with a non-negative $x$-coordinate. Then $(0,0)\in P$ and 
    if there exists an edge $G$ such that $m^G(0,0) > 0$, then it must be that $d_G = -(1, b)$ for some $b \in \NN$.
\end{proposition}

    \begin{proof}   
        That $(0,0) \in P$ follows immediately by the assumptions, since $$(x,h), (-n_F,y) \in P$$ for some $x \in [0,n_F]$ and some $y \geq -n_E$, and
\[
h := \max\{y \mid (x,y) \in P\} \geq n_E.
\]

        Let $G$ be an edge with $m^G(0,0) > 0$. Define $f_1 := \mut_F f$ and $P_1 := \newt(f_1)$. We choose the mutation $\mut_Ff$ (meaning we choose $g$ such that $f_1 = \mut_F f = \mut^g_{m^F} f$) in such a way that the edge $E$ remains an edge of $P_1$. Furthermore, we choose the mutation $\mut_E f_1$ such that the horizontal edge of $P_1$ lying on the line $(x = \tilde{w})$ with 
\begin{equation}\label{eq tilw}
\tw := \min\{x \mid (x,y) \in P\} = \min\{x \mid (x,y) \in P_1\}\in \ZZ,
\end{equation}
remains an edge of $\newt(\mut_Ef_1)$. Note that, by minimality of $f$, the lattice length of this edge is $\geq n_F$.

 We denote by $\tilde{G}$ the edge of $P_1$ on which $\psi_{m^F}(m^G)$ achieves its maximum.  
We define  
$$
n := \max\{y \mid (0,y) \in P\} = \max\{y \mid (0,y) \in P_1\}\geq 0.
$$

Let $d_G = -(a,b)$, where $a, b \in \mathbb{Z}$. Since $f$ is minimal, and $(-n_E,0)\in E$, and \eqref{eq prop 74} holds, and given that among the points of $P$ with maximal $y$-coordinate, there exists at least one with a non-negative $x$-coordinate, it follows immediately that $m^{G}(0,0) < 0$ unless both $a, b > 0$.

We distinguish the following two cases:

\textbf{CASE 1}: Let $b > a \geq 2$.  
We have  
\begin{equation}\label{eq nmutf1}
    n(\mut_F(f)) < n(f) - n_F + n_E + \frac{b-a}{a} \tw + n.
\end{equation}
Indeed, the edge of $P_1$ lying on the line $(x=n_F)$ has length $\ell(F) - n_F$, and the edge of $P_1$ lying on the line $(x=\tw)$ has length smaller than  
$$
n_E + \frac{b-a}{a} \tw + n,
$$  
since the line $y = \frac{b-a}{a}x + n$ has the same slope as the line passing through $\tilde{G}$, and we have  
$$
\ell(G) = \ell(\tilde{G}) \geq n_G, \quad -n_E = \min\{y \mid (x,y) \in P\},
$$  
from which the inequality follows. 
We observe that
\begin{equation}\label{eq nmutf2}
    n((\mut_E \circ \mut_F) f) < n(\mut_F f) - n_E + \frac{b-a}{a} n_F + n,
\end{equation}
since the highest $y$-coordinate of $P_1$ satisfies  
$$
\max\{y \mid (x,y) \in P_1\} \leq \frac{b-a}{a} n_F + n.
$$  
By our assumption, we must have  
\begin{equation}\label{nemff}
n((\mut_E \circ \mut_F) f) \geq n(f).
\end{equation}  
From \eqref{eq nmutf1} and \eqref{eq nmutf2}, we obtain  
\begin{equation}
    n((\mut_E \circ \mut_F) f) < \Big(n(f) - n_F + n_E + \frac{b-a}{a} \tw + n\Big) - n_E + \frac{b-a}{a} n_F + n.
\end{equation}
Simplifying the right-hand side, we get  
\begin{equation}
    n((\mut_E \circ \mut_F) f) < n(f) + 2n + \frac{b-a}{a} (\tw + n_F) - n_F.
\end{equation}
Thus by \eqref{nemff} we get 
$$
2n \geq n_F + \frac{b-a}{a} (-\tw - n_F),
$$
which leads to  
\begin{equation}\label{eq ang}
    an \geq \frac{an_F + (b-a)(-\tw - n_F)}{2}.
\end{equation}

On the other hand, we observe that  
$$
\frac{-\tw + n_F}{a} \geq \ell(G) > an,
$$  
where the latter inequality follows from the fact that $m^G(0,0) > 0$. Thus, using \eqref{eq ang}, we obtain  
$$
\frac{-\tw + n_F}{a} > \frac{an_F + (b-a)(-\tw - n_F)}{2}.
$$  
Since $b > a \geq 2$, this leads to a contradiction.

\textbf{CASE 2}: Let $a > b$.  
In this case, we have  
$$
n + \frac{b}{a} n_F \geq n_E,
$$  
since the highest $y$-coordinate of $P$ satisfies  
$$
\max\{y \mid (x,y) \in P\} \leq n + \frac{b}{a} n_F.
$$  
Thus, we deduce that  
$$
an \geq an_E - bn_F \geq n_E,
$$  
which is a contradiction since $n_G > an \geq n_E$, where the first inequality follows from $m^G(0,0)>0$.

Thus we conclude the proof. 
    \end{proof}

\begin{theorem}\label{th smoothable}
    Any Laurent polynomial $f\in \CC[N]\cong \CC[\ZZ^2]$ is mutation equivalent to a Laurent polynomial $g$ such that there exists a lattice point $v\in \newt(g)$ such that $m(v)\leq 0$ for every $m\in \cM(g)$.
\end{theorem}

\begin{proof}
    It is enough to prove that one of the following holds:
    \begin{enumerate}
        \item For a Laurent polynomial $f$, there exists a lattice point $v\in P:=\newt(f)$ such that $m(v) \leq 0$ for all $m\in \cM(f)$.
        \item The Laurent polynomial $f$ is not minimal, meaning that there exists a Laurent polynomial $h$ that is mutation equivalent to $f$ and satisfies $n(h) < n(f)$. 
    \end{enumerate}
    Note that if $\newt(g)$ is only a lattice point $v$, then clearly $m(v) \leq 0$ for all $m\in \cM(g)$.
    Assume that (1) and (2) do not hold. We pick an edge $E$ of $P$ such that  
    $$
    n_E = \max\{n_F \mid F\text{ is an edge of }P\},
    $$  
    and without loss of generality, we assume that $E = [(0,0), (\ell(E),0)]$. We write the coordinates as $(x,y) \in N_\RR \cong \RR^2$.  
    Let us define the height $h$ of $P$ to be the maximal $y$-coordinate in $P$:  
    \begin{equation}\label{eq hhh}
    h := \max\{y \in \NN \mid (x,y) \in P\}.
    \end{equation}
    By our assumption, we have $h \geq 2n_E$, since otherwise, we would have $$n(\mut_E(f)) < n(f).$$  
Let us define  
\begin{equation}\label{eq x0}
    x_0 := \min\{k\in \ZZ \mid (k,h) \in P\}.
\end{equation} 
We can assume that $0 \leq x_0 < h$, since otherwise there exists a $\GL_2(\ZZ)$-map that preserves $E$ and maps $x_0$ to the desired interval. Note that affine equivalence of Newton polygons induces an isomorphism of Laurent polynomials with the same coefficients. 

\textbf{CASE 1}: Let $0 \leq x_0 < \ell(E)$, and denote the lattice point  
$$
A := (x_0, n_E) \in P.
$$  
By the definition of $n_E$ (and since $h \geq 2n_E$ by assumption), we immediately see that  
$$
m^F(A) \leq 0 \quad \text{if} \quad d_F \not\in \{\pm (1,0), \pm (0,1)\},
$$
which implies the presence of a vertical edge $F$ (i.e., $d_F = \pm (0,1)$).

Thus, modulo an affine equivalence, we are in the setting of Proposition \ref{prop 74} and we also use the same notation: 
$(0,0) \in P = \newt(f)$, $E \subset (y = -n_E)$ is a horizontal edge of $P$, and $F \subset (x = n_F)$ is a vertical edge. Moreover, there exists an edge $G'$ with $d_{G'} = -(1,b')$, $b' \in \NN$, such that $m^{G'}(0,0) > 0$. For every such edge $G'$, we denote
$$
x_{G'} := \max\{x \mid m^{G'}(x,0) > 0\},
$$
and let $G$, with $d_G = -(1,b)$, be an edge with the maximal $x_G$. 

Recall the definitions of $h$ and $x_0$ from \eqref{eq hhh} and \eqref{eq x0}, respectively. In this case, we have $(x_0, -n_E) \in E \subset P$, and this is the main distinction with the Case 2 that we will consider later.

If $m^H(x_0,0) > 0$ for some edge $H$, then $H$ must either be equal to $F$ or parallel to $F$, otherwise this would contradict the definition of $n_E$.
Next, define
$$
f_1 := \mut_{F}f,
$$
and denote by $G_1$ the edge of $P_1 := \newt(f_1)$ on which $\psi_{m^F}(m^G)$ achieves its maximum.  Define
$$
h_1 := \max\{y\in \NN \mid (x,y) \in P_1\}, \quad x_1 := \min\{x \in \ZZ \mid (x,h_1) \in P_1\}.
$$

We assume that neither (1) nor (2) holds for $f$, and we will demonstrate that this leads to a contradiction by considering several possible cases. We choose $\mut_Ff$ and $(\mut_E\circ \mut_F)f$ as in the proof of Proposition \ref{prop 74}. Note that $x_1\leq x_0$, since $d_G=-(1,b)$, $b\geq 1$.

\textbf{Case a:} $b = 1$, meaning $d_G = -(1,1)$.

In this case, $P_1$ has two parallel edges, $E$ and $G_1$. Since $m^{G_1}(0,0) > 0$, $m^E(0,0) = 0$, and $f_1$ is both $m^{G_1}$-mutable and $m^E$-mutable, we conclude that $f_1$ is decomposable: 
$
f_1 = g_1 h,
$
where $\newt(h)$ is a horizontal line segment and $g_1$ is in a mutation equivalence class of $f_1$ (and $f$). Thus, we obtain  
$$
n(\mut_{\widetilde{F}} g_1) < n(f),
$$  
where $\widetilde{F}$ is the vertical edge of $P_1$ with minimal $x$-coordinate. This is a contradiction.

\textbf{Case b}: $b\geq 2$, which means $d_G=-(1,b)$.

We now distinguish the following subcases: 

\begin{enumerate}
\item There exists a lattice point $A_1 = (x,0)\in N\cong \ZZ^2$, with $x \in [0, x_1]$ (meaning that we also assume that $x_1\geq 0$), such that $m^{G_1}(A_1) < 0$ and $m^{F_1}(A_1) < 0$, where $F_1$ is an edge of $P_1$ lying on the line $(x = n_F)$ (noting that the edge $F$ of $P$ lies on the same line). In this case, there must exist an edge $H$ of $P_1$, with $d_H\in N\cong \ZZ^2$ lying in the second quadrant (i.e.\ the first coordinate of $d_H$ is negative and the second one is positive), such that $m^H(A_1) > 0$. Thus, 
\begin{equation}\label{eq nmutfmute}
n((\mut_E \circ \mut_F) f) < n(f) + (-n_F + n_E) - n_E + h_1 \leq n(f),
\end{equation}
since $m^H(A_1) > 0$ and thus $h_1 \leq n_F$, which implies that $n(\mut_F f) < n(f) + (-n_F + n_E)$. Indeed, this holds since $\max\{y \mid (\tilde{w},y) \in P_1\} \leq 0$, where recall $\tilde{w}$ from \eqref{eq tilw}, which holds because there is a point $(x_1,h_1)\in P_1$, with $h_1\leq n_F$ and $d_{G_1}=-(1,b-1)$ and thus the claim follows because $\tilde{w}\leq -n_F$ and $x_1\geq 0$.

\item It holds that $m^{G_1}(x_1,0) > 0$. In this case, we have
\begin{multline}\label{eq pom 74 pom}
n((\mut_G \circ \mut_E \circ \mut_F) f) <\\ n(f) + (-n_F + n_E) + (-n_E + n_G) + (-n_G + n_F) \leq n(f).
\end{multline}
Indeed, since $m^{G_1}(x_1,0) > 0$, we see that 
$$
n((\mut_E \circ \mut_F) f) < n(f) + (-n_F +n_E+ h_1-(b-1)n_F) - n_E + h_1,
$$ 
since $\max\{y \mid (\tilde{w},y) \in P_1\} \leq -n_E+n_E+ h_1-(b-1)n_F$, where recall $\tilde{w}$ from \eqref{eq tilw}. 
This implies that  $-bn_F + 2h_1\geq 0,$ since otherwise we obtain a contradiction. Since also $h_1\leq \frac{n_G}{b-1}$ (because $m^{G_1}(x_1,0) > 0$ with $d_{G_1}=-(1,b-1)$), we obtain 
$$\frac{2n_G}{b-1}\geq 2h_1\geq bn_F$$
and thus 
\begin{equation}\label{eq nggeq}
n_G\geq \frac{b(b-1)}{2}n_F.
\end{equation}
We have 
$h_1+n_F\geq h=\max\{y\in \NN\mid (x,y)\in P\}\geq n_E$, and thus $$n_F\geq n_E-h_1.$$
Inserting this into \eqref{eq nggeq} gives us 
\begin{equation}\label{eq ng2}
n_G\geq \frac{b(b-1)(n_E-h_1)}{2}.
\end{equation}
 We know that $h_1<\frac{n_G}{b-1}$ and thus $-h_1>-\frac{n_G}{b-1}$. Inserting this into \eqref{eq ng2} gives us that 
 $$2n_G+bn_G\geq b(b-1)n_E$$
 and thus
 $b\leq 2$, since otherwise $n_G>n_E$, which is a contradiction. 
Thus $b=2$, $d_{G_1} = -(1,1)$ and then
$$n((\mut_E \circ \mut_F)f) < n(f) + (-n_F + n_E) + (-n_E + n_G).$$
When we additionally mutate by $\mut_G$, we need to add $-n_G + n_F$, since the maximum distance between a point on $G_1$ and a point on $P_1$, in the direction of $c_{G_1}$, is achieved at the point $(n_F,0)$ and it is smaller than $n_G + n_F$. Thus we proved \eqref{eq pom 74 pom}.

\item Otherwise, we set $f_2 := \mut_{F_1}f_1$ with $P_2 = \newt(f_2)$ and proceed as in the previous step. If $b = 2$, we conclude the proof as in Case 1a. If $b \geq 3$, we distinguish two possibilities:

\begin{itemize}
\item if there exists a lattice point $A_2 = (x,0)$ with $x \in [0, x_2]$, where
$$
h_2 := \max\{y\in \ZZ \mid (x,y) \in P_2\}, \quad x_2 = \min\{x \in \ZZ \mid (x, h_2) \in P_2\},
$$
such that $m^{G_2}(A_2) < 0$ (where $G_2$ is an edge of $P_2$ on which $\psi_{m^{F_1}}(m^{G_1})$ achieves its maximum) and $m^{F_2}(A_2) < 0$ (where $F_2$ is an edge of $P_2$ lying on the line $(x = n_F)$), then we conclude as in the case (1) above: there exists an edge $H$ of $P_2$, with $d_H$ lying in the second quadrant, such that $m^H(A_2) > 0$. This yields
$$n((\mut_E \circ \mut_{F_1}) f_1) < n(f) + (-n_F - n_{F_1} + n_E) - n_E + h_2 \leq n(f).
$$
Here, $\mut_{F_1}f_1$ is chosen such that the edge $E$ is the same for both $\newt(f_1)$ and $\newt(\mut_{F_1}f_1)$.
\item It holds that $m^{G_2}(x_2,0) > 0$, where we conclude as in the case (2) above: 
let us denote $\tilde{f}:=(\mut_{F_1} \circ \mut_{F}) f$.
We see that 
\begin{equation}\label{pompom eq 75}
n((\mut_{F_1} \circ \mut_{F}) f) < n(f) + (-n_F-n_{F_1} + n_E),
\end{equation}
where note that the edge $E$ is the same for $f$ and $\tilde{f}$. Moreover, we choose $\mut_E\tilde{f}$ such that the edge $G_2$ is the same for $\newt(\tilde{f})$ and $\newt(\mut_E\tilde{f})$.
In the same way as we prove \eqref{eq pom 74 pom}, we also prove that
$$
n((\mut_{G_2} \circ \mut_E) \tilde{f}) < n(\tilde{f}) + (-n_E + n_G) + (-n_G + n_{F_1})<n(f),
$$
which is a contradiction. In the last inequality we used \eqref{pompom eq 75} and the fact that $n_G=n_{G_2}$.
\end{itemize}
We continue this procedure, noting that it eventually terminates, since $G_b$ is a horizontal line.

\end{enumerate}

\textbf{CASE 2}: Let $\ell(E) \leq x_0 < h$, and denote the lattice point  
$$
B := (\ell(E), n_E) \in P.
$$  
It is straightforward to verify that  
\begin{equation}\label{eq vaprhimb}
    m^F(B) \leq 0 \quad \text{if} \quad d_F \not\in \{\pm (0,1), -(1,1)\}.
\end{equation}
To prove \eqref{eq vaprhimb}, we need to analyze the cases $b > a > 0$ and $a > b > 0$, as the remaining cases follow immediately.  

First, assume that $d_F = -(a,b)$ with $b > a > 0$. We see that  
$$
m^F(B) \leq n_F - \langle (\ell(E), n_E), (b,-a) \rangle \leq n_F - n_E \leq 0,
$$
where the first inequality is obtained by computing the distance between $(0,0)$ and $B$ in the direction of $c_F = (b,-a)$, and the second inequality follows because $b > a$ and $\ell(E) \geq n_E$.  

If $a > b > 0$, we consider the point $\tilde{B} := (x_0, h) \in P$. We see that  
$$
m^F(B) \leq n_F - \langle B - \tilde{B}, (b,-a) \rangle
$$  
by computing the distance between $B$ and $\tilde{B}$ in the direction of $c_F = (b,-a)$. Since $B - \tilde{B} = -(b_1, b_2)$ with $b_1 < b_2$ and $b_2 \geq n_E$, we obtain  
$$
m^F(B) \leq n_F - n_E \leq 0.
$$

Thus, modulo an affine equivalence, we are in the setting of Proposition \ref{prop 74} and use the same notation: $(0,0) \in P = \newt(f)$, $E \subset (y = -n_E)$ is a horizontal edge of $P$, and $F \subset (x = n_F)$ is a vertical edge. We have $(0,-n_E)\in E$ and there exists an edge $G$ with $m^G(0,0) > 0$ and $d_G = -(1,b)$, $b \in \NN$, that has maximal $x_G$ as in Case 1. We also use the same notation for $F_1$ and $P_1$, where $P_1$ is the Newton polytope of $\mut_F f$, and $F_1$ is the edge of $P_1$ lying on the line $(x = n_F)$. In particular, we will also use the same choice of composition of mutations as in Case 1.

If $b = 1$, we conclude as in Case 1a. Thus, we assume $b \geq 2$.

Let $h$ and $x_0$ be as in \eqref{eq x0} and \eqref{eq hhh}, respectively. We define
$$
x_E := \max\{x \mid (x,y) \in E\}\in \NN.
$$

Let $(x_E + x_h, h_1)$ be a point of $P_1$ with maximal $y$-coordinate, where $x_h$ is chosen minimally such that this holds. If $x_h< 0$, we are in the setting of Case 1, since then $(x_0,-n_E)\in E\subset P$. Thus, let $x_h\geq 0$ and
we distinguish the following cases:

\begin{enumerate}
        
    \item There exists a lattice point $(\tilde{x}_1, 0) \in P_1$, with $\tilde{x}_1\in [0,x_E+x_h]$,  $m^G(\tilde{x}_1,0) < 0$ and $m^{F_1}(\tilde{x}_1,0) < 0$. We choose $\tilde{x}_1$ to be minimal with this property.

    \begin{enumerate}[label=\roman*)]
 
        \item If $\tilde{x}_1 > x_E$, then in particular, $m^{G_1}(x_E,0) > 0$, which implies that $d_G = -(1,2)$ (i.e., $b = 2$). The maximal $y$-coordinate of $P_1 \cap (x = x_E)$ is then $\leq n_G$, giving
        $$
        n(\mut_F f) = n(f_1) < n(f) + (-n_F + n_G),
        $$
        since $b = 2$ and thus the maximal $y$-coordinate of $P_1 \cap (x = x_E - \ell(E))$ is $\leq -n_E + n_G$, where note that $d_{G_1} = -(1,1)$, from which we see that $h_1\leq n_G+x_h$. Thus
        $$n((\mut_E \circ \mut_F) f) < n(f) + (-n_F + n_G) + (-n_E + n_G + x_h),
        $$
        and 
        \begin{multline}
        n((\mut_G \circ \mut_E \circ \mut_F) f) <\\ n(f) + (-n_F + n_G) + (-n_E + n_G + x_h) + (-n_G + n_F - x_E - x_h) \leq n(f),
        \end{multline}
      a contradiction.

        \item  $\tilde{x}_1\leq x_E$.  In this case, there exists an edge $H$ of $P_1$ such that $m^H(\tilde{x}_1,0)>0$ with $d_H$ lying in the second quadrant. Thus $h_1\leq n_F$ and it holds that 
        $$
         n((\mut_E \circ \mut_F) f) < n(f)+(-n_F+n_G)+(-n_E+h_1)\leq n(f),
         $$
        as in \eqref{eq nmutfmute}, which is a contradiction.

    \end{enumerate}

\item $m^{G_1}(x_E+x_h,0)>0$. We have
\begin{multline}
n((\mut_G\circ \mut_E \circ \mut_{F}) f) < \\n(f) + (-n_F+ n_E) + (-n_E+n_G)+(-n_G+n_F)= n(f),
\end{multline}
   a contradiction.

\item If the lattice point $(\tilde{x}_1,0)\in P_1$ with $m^G(\tilde{x}_1,0)<0$ and $m^{F_1}(\tilde{x}_1,0)<0$ does not exists for $\tilde{x}_1\in [0,x_E]$, and $m^{G_1}(x_E+x_h,0)<0$, then we define $f_2:=\mut_{F_1}f_1$ and we repeat the above procedure.
As in the Case 1, we see that this procedure eventually terminates since $G_b$ is a horizontal line.
\end{enumerate}
\end{proof}

\begin{example}\label{ex 76}
The first polygon in Figure \ref{fig pom1} is the Newton polytope of a Laurent polynomial $g_1$ that we now describe. The red dot denotes the origin $(0,0)$. Let the edges be defined as $E = \conv\{(-2,-3), (1,-3)\}$, $F = \conv\{(2,-2), (2,4)\}$, and $G = \conv\{(-2,-2), (2,4)\}$. Assume that $g_1$ is $m^E = m^E_{3,1}$-mutable, $m^F = m^F_{2,1}$-mutable, and $m^G = m^G_{2,1}$-mutable. We see that $d_G = -(2,3)$ and $m^G(0,0) = 0$, which is consistent with Proposition \ref{prop 74}.

The second polygon in Figure \ref{fig pom1} is the Newton polytope of a Laurent polynomial $g_2$, with $E = \conv\{(-2,-3), (1,-3)\}$, $F = \conv\{(2,-2), (2,0)\}$, and $G = \conv\{(-2,-3), (1,3)\}$. Assume that $g_2$ is $m^E = m^E_{3,1}$-mutable and $m^F = m^F_{2,1}$-mutable.
Thus $n_E=3$ and $n_F=2$ and we are in Case 1 part of the proof of Theorem \ref{th smoothable} since $(x_0,-n_E)=(1,-3)\in E$. 
The Newton polytope of the Laurent polynomial $\mut_F g_2$ is the third polytope presented in Figure \ref{fig pom1}. If $g_2$ is additionally $m^G = m^G_{2,1}$-mutable, then the lattice point $v = (1,0) \in \newt(\mut_F g_2)$ satisfies $m(v) \leq 0$ for all $m \in \cM(\mut_F g_2)$. In this case $x_1=1$, $h_1=2$ and we are in Case 1b(1) part of the proof.

If $g_2$ is aditionally $m^G = m^G_{3,1}$-mutable, then we are in both Case 1b(2) and Case 1b(3) part of the proof. 
We have
$$
n((\mut_E \circ \mut_F)g_2) = n(g_2) - 1
$$
and
$$
n((\mut_{G} \circ \mut_{E} \circ \mut_{F})g_2) = n(g_2) - 4.
$$
The last polytope in Figure \ref{fig pom1} is $\newt((\mut_E\circ \mut_F)(g_2))$.

The first polygon in Figure \ref{fig pom2} is the Newton polytope of a Laurent polynomial $f$ that we now describe. Let $$E = \conv\{(-2,-3), (1,-3)\},~F = \conv\{(2,-2), (2,5)\},$$ and $G = \conv\{(-2,-3), (2,5)\}$. Assume that $f$ is $m^E = m^E_{3,1}$-mutable and $m^F = m^F_{2,1}$-mutable. Thus $n_E=3$, $n_F=2$. We are in Case 2 of the proof of Theorem \ref{th smoothable}, since $(x_0,-n_E)=(2,-3)\not\in E$. 

The Newton polytope of $f_1 = \mut_F f$ is the second polygon in Figure \ref{fig pom2}. Using the notation of the proof of Theorem \ref{th smoothable}, $F_1$ is the edge of $\newt(f_1)$ lying on the line $(x = 2)$. Moreover, $x_E=1$, $x_h=1$ and $h_1=3$. If $m^{F_1} = m^{F_1}_{1,1}$ and $m^G=m^G_{2,1}$, then the point $(1,0)\in P_1=\newt(f_1)$ satisfies $m(1,0) = \varphi_m(1,0) \leq 0$ for all $m \in \cM(f_1)$, in particular, we are in Case 2(1) part of the proof with $\tilde{x}=1$. If $m^{F_1} = m^{F_1}_{2,1}$ and $m^G=m^G_{2,1}$, then we are in Case 2(3) and the third polygon in Figure \ref{fig pom2} is the Newton polytope of $f_2 = \mut_{F_1} f_1$. 
Clearly, $f_2$ is decomposable, as it is both $m^{G_2} = m^{G_2}_{2,1}$-mutable and $m^E$-mutable, where $G_2=\conv\{(-2,1),(2,1)\}$. Finally, we see that
$
n(\mut_E f_2) = n(f) - 2.
$

\end{example}
    \begin{figure}[htb]
\begin{tikzpicture}[scale=0.5]
\draw[thick,  color=black]  
  (0,0) -- (3,0) -- (4,1) --  (4,7) -- (0,1) -- cycle;
\fill[thick,  color=red]
  (2,3) circle (2.5pt); 
\fill[thick,  color=black]
  (0,0) circle (2.5pt) 
  (0,1) circle (2.5pt)
  (1,0) circle (2.5pt)
  (1,1) circle (2.5pt)
  (1,2) circle (2.5pt)
  (2,0) circle (2.5pt) 
  (2,1) circle (2.5pt)
  (2,2) circle (2.5pt)
 (2,4) circle (2.5pt)
 (3,0) circle (2.5pt) 
  (3,1) circle (2.5pt)
 (3,2) circle (2.5pt) 
  (3,3) circle (2.5pt)
 (3,4) circle (2.5pt) 
  (3,5) circle (2.5pt)
  (4,1) circle (2.5pt)
(4,2) circle (2.5pt)
(4,3) circle (2.5pt)
  (4,4) circle (2.5pt)
(4,5) circle (2.5pt)
(4,6) circle (2.5pt)
  (4,7) circle (2.5pt);
\end{tikzpicture}
~~~~~~~~~~~
\begin{tikzpicture}[scale=0.5]
\draw[thick,  color=black]  
  (0,0) -- (3,0) -- (4,1) --  (4,3) -- (3,6) -- cycle;
\fill[thick,  color=red]
  (2,3) circle (2.5pt); 
\fill[thick,  color=black]
  (0,0) circle (2.5pt) 
  (1,0) circle (2.5pt)
  (1,1) circle (2.5pt)
  (1,2) circle (2.5pt)
  (2,0) circle (2.5pt) 
  (2,1) circle (2.5pt)
  (2,2) circle (2.5pt)
 (2,4) circle (2.5pt)
 (3,0) circle (2.5pt) 
  (3,1) circle (2.5pt)
 (3,2) circle (2.5pt) 
  (3,3) circle (2.5pt)
 (3,4) circle (2.5pt) 
  (3,5) circle (2.5pt)
  (4,1) circle (2.5pt)
(4,2) circle (2.5pt)
(4,3) circle (2.5pt);
\end{tikzpicture}
~~~~~~~~~~~
\begin{tikzpicture}[scale=0.5]
\draw[thick,  color=black]  
  (0,0) -- (3,0) -- (4,1) --  (3,5) -- (0,2) -- cycle;
\fill[thick,  color=red]
  (2,3) circle (2.5pt); 
\fill[thick,  color=black]
  (0,0) circle (2.5pt) 
  (0,1) circle (2.5pt)
 (0,2) circle (2.5pt)
  (1,0) circle (2.5pt)
  (1,1) circle (2.5pt)
  (1,2) circle (2.5pt)
  (1,2) circle (2.5pt)
 (1,3) circle (2.5pt)
  (2,0) circle (2.5pt) 
  (2,1) circle (2.5pt)
  (2,2) circle (2.5pt)
 (2,4) circle (2.5pt)
 (3,0) circle (2.5pt) 
  (3,1) circle (2.5pt)
 (3,2) circle (2.5pt) 
  (3,3) circle (2.5pt)
 (3,4) circle (2.5pt) 
  (3,5) circle (2.5pt)
  (4,1) circle (2.5pt) ;
\end{tikzpicture}
~~~~~~~~~~~
\begin{tikzpicture}[scale=0.5]
\draw[thick,  color=black]  
  (0,0) -- (2,1) -- (5,5) --  (3,5) -- (0,2) -- cycle;
\fill[thick,  color=red]
  (2,3) circle (2.5pt); 
\fill[thick,  color=black]
  (0,0) circle (2.5pt) 
  (0,1) circle (2.5pt)
 (0,2) circle (2.5pt)
  (1,1) circle (2.5pt)
  (1,2) circle (2.5pt)
  (1,2) circle (2.5pt)
 (1,3) circle (2.5pt)
  (2,1) circle (2.5pt)
  (2,2) circle (2.5pt)
 (2,4) circle (2.5pt)
  (3,3) circle (2.5pt)
 (3,4) circle (2.5pt) 
  (3,5) circle (2.5pt)
   (4,4) circle (2.5pt)
    (4,5) circle (2.5pt) 
     (5,5) circle (2.5pt) ;
\end{tikzpicture}
\caption{Newton polytopes of $g_1$, $g_2$, $\mut_Fg_2$ and $(\mut_E\circ \mut_F)g_2$.}
\label{fig pom1}
\end{figure}

\begin{figure}[htb]
 \begin{tikzpicture}[scale=0.5]
\draw[thick,  color=black]  
  (0,0) -- (3,0) -- (4,1) --  (4,8) -- cycle;
\fill[thick,  color=red]
  (2,3) circle (2.5pt); 
\fill[thick,  color=black]
  (0,0) circle (2.5pt) 
  (1,0) circle (2.5pt)
  (1,1) circle (2.5pt)
  (1,2) circle (2.5pt)
  (2,0) circle (2.5pt) 
  (2,1) circle (2.5pt)
  (2,2) circle (2.5pt)
 (2,4) circle (2.5pt)
 (3,0) circle (2.5pt) 
  (3,1) circle (2.5pt)
 (3,2) circle (2.5pt) 
  (3,3) circle (2.5pt)
 (3,4) circle (2.5pt) 
  (3,5) circle (2.5pt)
  (3,6) circle (2.5pt)
  (4,1) circle (2.5pt)
(4,2) circle (2.5pt)
(4,3) circle (2.5pt)
  (4,4) circle (2.5pt)
(4,5) circle (2.5pt)
(4,6) circle (2.5pt)
  (4,7) circle (2.5pt);
\end{tikzpicture}
~~~~~~~~
 \begin{tikzpicture}[scale=0.5]
\draw[thick,  color=black]  
  (0,0) -- (3,0) -- (4,1) --  (4,6) --  (0,2)  -- cycle;
\fill[thick,  color=red]
  (2,3) circle (2.5pt); 
\fill[thick,  color=black]
  (0,0) circle (2.5pt) 
   (0,1) circle (2.5pt)
    (0,2) circle (2.5pt)
  (1,0) circle (2.5pt)
  (1,1) circle (2.5pt)
  (1,2) circle (2.5pt)
   (1,3) circle (2.5pt)
  (2,0) circle (2.5pt) 
  (2,1) circle (2.5pt)
  (2,2) circle (2.5pt)
 (2,4) circle (2.5pt)
 (3,0) circle (2.5pt) 
  (3,1) circle (2.5pt)
 (3,2) circle (2.5pt) 
  (3,3) circle (2.5pt)
 (3,4) circle (2.5pt) 
  (3,5) circle (2.5pt)
  (4,1) circle (2.5pt)
(4,2) circle (2.5pt)
(4,3) circle (2.5pt)
  (4,4) circle (2.5pt)
(4,5) circle (2.5pt)
(4,6) circle (2.5pt);
\end{tikzpicture}
~~~~~~~~
 \begin{tikzpicture}[scale=0.5]
\draw[thick,  color=black]  
  (0,0) -- (3,0) -- (4,1) --  (4,4) --  (0,4)  -- cycle;
\fill[thick,  color=red]
  (2,3) circle (2.5pt); 
\fill[thick,  color=black]
  (0,0) circle (2.5pt) 
   (0,1) circle (2.5pt)
    (0,2) circle (2.5pt)
    (0,3) circle (2.5pt)
    (0,4) circle (2.5pt)
  (1,0) circle (2.5pt)
  (1,1) circle (2.5pt)
  (1,2) circle (2.5pt)
   (1,3) circle (2.5pt)
      (1,4) circle (2.5pt)
  (2,0) circle (2.5pt) 
  (2,1) circle (2.5pt)
  (2,2) circle (2.5pt)
 (2,4) circle (2.5pt)
 (3,0) circle (2.5pt) 
  (3,1) circle (2.5pt)
 (3,2) circle (2.5pt) 
  (3,3) circle (2.5pt)
 (3,4) circle (2.5pt) 
  (4,1) circle (2.5pt)
(4,2) circle (2.5pt)
(4,3) circle (2.5pt)
  (4,4) circle (2.5pt);
\end{tikzpicture}
\caption{Newton polytopes of $f$, $f_1$ and $f_2$.}
\label{fig pom2}
\end{figure}

\begin{theorem}\label{th main3}
    For every Laurent polynomial $f\in \CC[N]\cong \CC[\ZZ^2]$, the pair $$(X_{\newt(f)},\partial X_{\newt(f)})$$ is unobstructed in $\{t_m \mid m \in \cM(f)\}$.
\end{theorem}

\begin{proof}
    By Theorem \ref{th smoothable} we know that any Laurent polynomial $f\in \CC[N]\cong \CC[\ZZ^2]$ is mutable to a Laurent polynomial $g$ for which there exists a lattice point $v \in \newt(g)$ such that $m(v) \leq 0$ for all $m \in \cM(g)$. By 
     \cite[Corollary 5.4]{AS98}, using that  
    $
    T^2_{X_{\newt(g)}} = T^2_{(X_{\newt(g)},\partial X_{\newt(g)})},
    $  
    we see that
    \begin{equation}\label{eq claimf}
        T^2_{(X_{\newt(g)},\partial X_{\newt(g)})} \Big(-\sum_{m\in \cM(g)} k_m m \Big) = 0, \quad \text{for every } k_m \in \NN.           
    \end{equation}
 Thus, $(X_{\newt(g)}, \partial X_{\newt(g)})$ is unobstructed in $\{ t_m \mid m \in \cM(g) \}$. By Theorems~\ref{th main 1}, \ref{th main 2}, and Lemma~\ref{lem 61}, it follows that $(X_{\newt(f)}, \partial X_{\newt(f)})$ is unobstructed in $$\{ t_m \mid m \in \cM(f) \},$$ since $g$ can be reached from $f$ by finitely many mutations. Hence, we can choose a generating set of $S_{\newt(f)}$ such that, after each mutation, the resulting generating set satisfies \eqref{eq mutmut}.
\end{proof}

\begin{definition}\label{def maxmut}
    We say that a Laurent polynomial is \emph{maximally mutable} if there does not exist a Laurent polynomial $g$ with $\cM(f) \subsetneqq \cM(g)$.  
\end{definition}

\begin{definition}\label{def 0mut}
    If $f$ is mutation equivalent to a Laurent polynomial $g$ with $\newt(g)$ a point, we say that $f$ is \emph{$0$-mutable}.
\end{definition}

\begin{remark}
    The notion that $f$ is $0$-mutable was introduced in \cite{CFP22}.  
    Note that Theorem \ref{th smoothable} in particular reproves that $f$ is rigid maximally mutable if and only if $f$ is $0$-mutable (see \cite[Theorem 3.5]{CFP22}), using only convex geometry tools. 
\end{remark}

\begin{example}\label{ex 78}
In Figure \ref{fig 5} we presented all three possible $0$-mutable Laurent polynomials that have Newton polytope equal to $P=\conv\{(0,0),(4,0),(0,5)\}$ and in Figure \ref{fig 6} we presented maximally mutable Laurent polynomial 
$$f=(1+x)^4+y(5-15x^2-10x^3)+y^2(10-12x-22x^2)+y^3(10-8x)+5y^4+y^5,$$ 
   that is not $0$-mutable and $\newt(f)=P$. The Laurent polynomial 
   $f$ is $m_1:=(0,-1,3)$-mutable and $m_2:=(-2,0,4)$-mutable and clearly it does not exists $m\in \tM\setminus \cM(f)$ and a Laurent polynomial $h$ such that 
   $$\{m_1, m_2, m\}\subset \cM(h).$$
   For $g_1=1+x$ it holds that $f_1:=\mut_{m_1}^{g_1}f$ is the second Laurent polynomial in Figure \ref{fig 6}. Finally, for $g_2=1+y$ we have that 
   $$f_2:=\mut_{m_2}^{g_2}f_1=1+x+y-12xy+xy^2+2xy^3+x^2y^5,$$
   which is the last Laurent polynomial presented in Figure \ref{fig 6}. This Laurent polynomial satisfies the property (1) of Theorem \ref{th smoothable}: indeed, the lattice point $(1,1)$ that correspond to $xy$ has the desired property.
    \begin{figure}[htb]
\begin{tikzpicture}[scale=0.8]
\draw[thick,  color=black]  
  (0,0) -- (4,0) -- (0,5) -- cycle;
\fill[thick,  color=black]
  (0,0) circle (2.5pt) 
  (0,1) circle (2.5pt)
  (0,2) circle (2.5pt)
  (0,3) circle (2.5pt)
  (0,4) circle (2.5pt)
  (0,5) circle (2.5pt)
  (1,0) circle (2.5pt)
  (1,1) circle (2.5pt)
  (1,2) circle (2.5pt)
  (1,3) circle (2.5pt)
  (2,0) circle (2.5pt) 
  (2,1) circle (2.5pt)
  (2,2) circle (2.5pt)
 (3,0) circle (2.5pt) 
  (3,1) circle (2.5pt)
  (4,0) circle (2.5pt);
     \node[anchor=south west] at (0,0) {$1$};
    \node[anchor=south west] at (0,1) {$5$};
    \node[anchor=south west] at (0,2) {$10$};
    \node[anchor=south west] at (0,3) {$10$};
     \node[anchor=south west] at (0,4) {$5$};
       \node[anchor=south west] at (0,5) {$1$};
       \node[anchor=south west] at (1,0) {$4$};
    \node[anchor=south west] at (1,1) {$12$};
    \node[anchor=south west] at (1,2) {$12$};
    \node[anchor=south west] at (1,3) {$4$};
     \node[anchor=south west] at (2,0) {$6$};
    \node[anchor=south west] at (2,1) {$12$};
    \node[anchor=south west] at (2,2) {$6$};
 \node[anchor=south west] at (3,0) {$4$};
    \node[anchor=south west] at (3,1) {$4$};
    \node[anchor=south west] at (4,0) {$1$};
\end{tikzpicture}
~~~~
\begin{tikzpicture}[scale=0.8]
\draw[thick,  color=black]  
  (0,0) -- (4,0) -- (0,5) -- cycle;
\fill[thick,  color=black]
  (0,0) circle (2.5pt) 
  (0,1) circle (2.5pt)
  (0,2) circle (2.5pt)
  (0,3) circle (2.5pt)
  (0,4) circle (2.5pt)
  (0,5) circle (2.5pt)
  (1,0) circle (2.5pt)
  (1,1) circle (2.5pt)
  (1,2) circle (2.5pt)
  (1,3) circle (2.5pt)
  (2,0) circle (2.5pt) 
  (2,1) circle (2.5pt)
  (2,2) circle (2.5pt)
 (3,0) circle (2.5pt) 
  (3,1) circle (2.5pt)
  (4,0) circle (2.5pt);
     \node[anchor=south west] at (0,0) {$1$};
    \node[anchor=south west] at (0,1) {$5$};
    \node[anchor=south west] at (0,2) {$10$};
    \node[anchor=south west] at (0,3) {$10$};
     \node[anchor=south west] at (0,4) {$5$};
       \node[anchor=south west] at (0,5) {$1$};
       \node[anchor=south west] at (1,0) {$4$};
    \node[anchor=south west] at (1,1) {$15$};
    \node[anchor=south west] at (1,2) {$20$};
    \node[anchor=south west] at (1,3) {$10$};
     \node[anchor=south west] at (2,0) {$6$};
    \node[anchor=south west] at (2,1) {$15$};
    \node[anchor=south west] at (2,2) {$10$};
 \node[anchor=south west] at (3,0) {$4$};
    \node[anchor=south west] at (3,1) {$5$};
    \node[anchor=south west] at (4,0) {$1$};
\end{tikzpicture}
~~~~
\begin{tikzpicture}[scale=0.8]
\draw[thick,  color=black]  
  (0,0) -- (4,0) -- (0,5) -- cycle;
\fill[thick,  color=black]
  (0,0) circle (2.5pt) 
  (0,1) circle (2.5pt)
  (0,2) circle (2.5pt)
  (0,3) circle (2.5pt)
  (0,4) circle (2.5pt)
  (0,5) circle (2.5pt)
  (1,0) circle (2.5pt)
  (1,1) circle (2.5pt)
  (1,2) circle (2.5pt)
  (1,3) circle (2.5pt)
  (2,0) circle (2.5pt) 
  (2,1) circle (2.5pt)
  (2,2) circle (2.5pt)
 (3,0) circle (2.5pt) 
  (3,1) circle (2.5pt)
  (4,0) circle (2.5pt);
     \node[anchor=south west] at (0,0) {$1$};
    \node[anchor=south west] at (0,1) {$5$};
    \node[anchor=south west] at (0,2) {$10$};
    \node[anchor=south west] at (0,3) {$10$};
     \node[anchor=south west] at (0,4) {$5$};
       \node[anchor=south west] at (0,5) {$1$};
       \node[anchor=south west] at (1,0) {$4$};
    \node[anchor=south west] at (1,1) {$12$};
    \node[anchor=south west] at (1,2) {$12$};
    \node[anchor=south west] at (1,3) {$4$};
     \node[anchor=south west] at (2,0) {$6$};
    \node[anchor=south west] at (2,1) {$9$};
    \node[anchor=south west] at (2,2) {$3$};
 \node[anchor=south west] at (3,0) {$4$};
    \node[anchor=south west] at (3,1) {$2$};
    \node[anchor=south west] at (4,0) {$1$};
\end{tikzpicture}
\caption{$0$-mutable Laurent polynomials.}
\label{fig 5}
\end{figure}

\begin{figure}[htb]
\begin{tikzpicture}[scale=0.8]
\draw[thick,  color=black]  
  (0,0) -- (4,0) -- (0,5) -- cycle;
\fill[thick,  color=black]
  (0,0) circle (2.5pt) 
  (0,1) circle (2.5pt)
  (0,2) circle (2.5pt)
  (0,3) circle (2.5pt)
  (0,4) circle (2.5pt)
  (0,5) circle (2.5pt)
  (1,0) circle (2.5pt)
  (1,1) circle (2.5pt)
  (1,2) circle (2.5pt)
  (1,3) circle (2.5pt)
  (2,0) circle (2.5pt) 
  (2,1) circle (2.5pt)
  (2,2) circle (2.5pt)
 (3,0) circle (2.5pt) 
  (3,1) circle (2.5pt)
  (4,0) circle (2.5pt);
     \node[anchor=south east] at (0,0) {$1$};
    \node[anchor=south east] at (0,1) {$5$};
    \node[anchor=south east] at (0,2) {$10$};
    \node[anchor=south east] at (0,3) {$10$};
     \node[anchor=south east] at (0,4) {$5$};
       \node[anchor=south east] at (0,5) {$1$};
       \node[anchor=south east] at (1,0) {$4$};
    \node[anchor=south east] at (1,1) {$0$};
    \node[anchor=south east] at (1,2) {$-12$};
    \node[anchor=south east] at (1,3) {$-8$};
     \node[anchor=south east] at (2,0) {$6$};
    \node[anchor=south east] at (2,1) {$-15$};
    \node[anchor=south east] at (2,2) {$-22$};
 \node[anchor=south east] at (3,0) {$4$};
    \node[anchor=south east] at (3,1) {$-10$};
    \node[anchor=south east] at (4,0) {$1$};
\end{tikzpicture}
~~~~~
\begin{tikzpicture}[scale=0.8]
\draw[thick,  color=black]  
  (0,0) -- (1,0) -- (2,5) -- (0,5) --cycle;
\fill[thick,  color=black]
  (0,0) circle (2.5pt) 
  (0,1) circle (2.5pt)
  (0,2) circle (2.5pt)
  (0,3) circle (2.5pt)
  (0,4) circle (2.5pt)
  (0,5) circle (2.5pt)
  (1,0) circle (2.5pt)
  (1,1) circle (2.5pt)
  (1,2) circle (2.5pt)
  (1,3) circle (2.5pt)
  (1,4) circle (2.5pt)
  (1,5) circle (2.5pt)
  (2,5) circle (2.5pt);
     \node[anchor=south east] at (0,0) {$1$};
    \node[anchor=south east] at (0,1) {$5$};
    \node[anchor=south east] at (0,2) {$10$};
    \node[anchor=south east] at (0,3) {$10$};
     \node[anchor=south east] at (0,4) {$5$};
       \node[anchor=south east] at (0,5) {$1$};
       \node[anchor=south east] at (1,0) {$1$};
    \node[anchor=south east] at (1,1) {$-10$};
    \node[anchor=south east] at (1,2) {$-22$};
    \node[anchor=south east] at (1,3) {$-8$};
      \node[anchor=south east] at (1,4) {$5$};
        \node[anchor=south east] at (1,5) {$2$};
  \node[anchor=south east] at (2,5) {$1$};
\end{tikzpicture}
~~~~~
\begin{tikzpicture}[scale=0.8]
\draw[thick,  color=black]  
  (0,0) -- (1,0) -- (2,5) -- (0,1) --cycle;
\fill[thick,  color=black]
  (0,0) circle (2.5pt) 
  (0,1) circle (2.5pt)
  (1,0) circle (2.5pt)
  (1,1) circle (2.5pt)
  (1,2) circle (2.5pt)
  (1,3) circle (2.5pt)
  (2,5) circle (2.5pt);
     \node[anchor=south east] at (0,0) {$1$};
    \node[anchor=south east] at (0,1) {$1$};
       \node[anchor=south east] at (1,0) {$1$};
    \node[anchor=south east] at (1,1) {$-12$};
    \node[anchor=south east] at (1,2) {$1$};
    \node[anchor=south east] at (1,3) {$2$};
  \node[anchor=south east] at (2,5) {$1$};
\end{tikzpicture}
\caption{Maximally mutable Laurent polynomial and its two mutations.}
\label{fig 6}
\end{figure}
\end{example}

From Theorem \ref{th main3}, we thus obtain the following flat map:  
\begin{equation}\label{eq picmf}
    \pi_{\cM(f)}: \spec R_{\cM(f)} \to \spec \CC[[t_m\mid m\in \cM(f)]],
\end{equation}
where $R_{\cM(f)}=\CC[\bfx,u][[\bft_{\cM(f)}]]/(F_\bfk(\bfx,u,\bft_{\cM(f)})\mid \bfk\in \NN^{\tilde{r}})$,  $\pi^{-1}(0) = X_{\newt(f)}$, and the restriction to $\CC[[t_m]]$ coincides with the deformation  
\begin{equation}
    \spec \CC[\bfx,u,t_m] / (F_\bfk(\bfx,u,t_m) \mid \bfk \in \NN^{\tr}) \to \spec \CC[[t_m]],
\end{equation}
from Corollary \ref{cor 3.3}.

\begin{theorem}\label{th smooth}
    The general fibre of the unobstructed deformation $\pi_{\cM(f)}$ of $X_{\newt(f)}$ is smooth if and only if $f$ is $0$-mutable.
\end{theorem}

\begin{proof}
    If $f$ is $0$-mutable, then the general fibre is smooth by Theorems \ref{main th 1 un} and \ref{th main3}.  

    In the other direction, by the same theorems, we need to show that the general fibre of the unobstructed deformation of $X_{\newt(g)}$ over  
    $
    \CC[[t_m \mid m \in \cM(g)]]
    $ is singular, 
    where $g$ is such that there exists a lattice point $v \in G:=\newt(g)$ satisfying  
    $
    m(v) \leq 0, \text{for all } m \in \cM(g).
    $  
    Let $H_G := \{\s_1, \dots, s_{\tr}, R^*\}$ be a generating set of $S_G$, and choose  
    $$
    h_1, h_2 \in H_G
    $$  
    such that for any other element $h \in \{\s_1, \dots, s_{\tr}\}$, with $h \neq h_i$, $i = 1,2$, it holds that  
    $$
    h(v) \geq k := \max\{h_1(v), h_2(v)\}.
    $$  

For any equation $F_\bfk$, we observe that it is not possible to have a monomial of the form $x_1 T$, $x_2 T$, or $u T$, where  
$$
T = a\cdot \prod_{m \in \cM(g)} t_m^{n_m}
$$  
is a product of deformation parameters for some $n_m \in \NN$, and where $x_i$ corresponds to $h_i$ for $i = 1,2$, and $a\in \CC$.
 Indeed, this is not possible due to degree considerations, since $s_\bfk(v) > k$ and every deformation parameter $t_m$ has a non-positive value on $v$, i.e.,  
    $
    m(v) \leq 0.
    $  
\end{proof}

For every $m \in \cE(P)$, we choose a line segment $Q \subset (m=0)$ of lattice length $1$ and denote by $P_m := P_{(m,Q)}$ the mutation polytope.  
Moreover, we define  
$$
\mut_{m_2} \circ \mut_{m_1}(P) := \mut_{m_2} P_{m_1} = (P_{m_1})_{m_2},
$$  
and similarly for  
$$
\mut_{m_n} \circ \cdots \circ \mut_{m_1}(P).
$$  
We write $(P, \cM) \sim (\tilde{P}, \tilde{\cM})$ if there exist $m_i \in \cM_i$ for $i = 1, \dots, n$, where $\cM_i := \psi_{m_i}(\cM_{i-1})$ with $\cM_0 := \cM$, such that  
$$
\mut_{m_n} \circ \cdots \circ \mut_{m_1}(P) = \tilde{P} \quad \text{and} \quad \cM_n = \tilde{\cM}.
$$  

\begin{lemma}\label{lem 710}
    If $(\tilde{P},\tilde{\cM}) \sim (P,\cM)$, such that $\tilde{P}$ is not $m$-mutable for some $m \in \tilde{\cM}$. Then $X_P$ is not unobstructed in $\{t_m\mid m\in \cM\}$.
\end{lemma}
\begin{proof}
    This follows immediately from Theorem \ref{main th 1 un}.
\end{proof}

As a corollary we get the following theorem.

\begin{theorem}\label{th dis kon}
     If $f$ is maximally mutable and $\cM$ is such that $\cM(f) \subsetneqq \cM\subset \cE(\newt(f))$, then $X_{\newt(f)}$ is not unobstructed in $\{t_m\mid m\in \cM\}$. 
\end{theorem}
\begin{proof}
    Let $f$ be maximally mutable and $m\in \cM:=\cM(f)\cup \{r\}$ for some $r\in \cE(\newt(f))\setminus \cM(f)$. From the proof of Theorem \ref{th smoothable}, we easily see that $(\newt(f),\cM)\sim (\tilde{P},\tilde{\cM})$, 
    such that one of the following holds: 
\begin{enumerate}
    \item there exists $m\in \tilde{\cM}$ such that $\tilde{P}$ is not $m$-mutable;
    \item there exists a lattice point $v\in \tilde{P}$ such that $m(v)\leq 0$ for all $m\in \tilde{\cM}$. 
\end{enumerate} 
    We will show that (2) is not possible by proving that $\tilde{\cM}\subset \cM(g)$, for some Laurent polynomial $g$ with $\newt(g)=\tilde{P}$, which contradicts the assumption that $f$ is maximally mutable. 
If such a point $v$ exists, we can subdivide $\tilde{P}$ into triangles
$$
T_E = \operatorname{conv}\{v, E \mid E \text{ is an edge of } \tilde{P}\}.
$$
For every edge $E$ of $\tilde{P}$, we can clearly choose values at the lattice points on $T_E$, which yields a Laurent polynomial $g$ with $\newt(g) = \tilde{P}$  
that is $m^E_{n,k}$-mutable for every $m^E_{n,k} \in \tilde{\mathcal{M}}$. Note that there is no issue in choosing values on the intersection $T_E \cap T_F$ for two edges $E$ and $F$ with nontrivial intersection,  
and thus we have constructed a Laurent polynomial 
 $g$ with $\newt(g) = \tilde{P}$ that is $m$-mutable for every $m\in \tilde{\mathcal{M}}$.
\end{proof}

Thus, we conclude the proof of Theorem \ref{thmm} from the introduction.

\section{The miniversal components}\label{comp sec}

Let $X=X_P$ be three-dimensional affine Gorenstein toric variety.
Proposition \ref{t1 prop} implies that 
 $$T^1_{\paar}=\bigoplus_{k\in \NN}T^1_{\paar}(-kR^*)\bigoplus_{m\in \cT} T^1_{\paar}(-m),$$
where $\dim_\CC T^1_{\paar}(-m)=1$, if $$m\in \cT:=\{m\in \tM\mid \text{if $P$ is $m$-mutable and } \max_{v \in P} \varphi_m(v) \geq 1\}.$$

In \cite{Fil23}, we analysed deformations induced by the elements of $$\bigoplus_{k \in \mathbb{N}} T^1_{\paar}(-k R^*).$$ These deformations are related to Minkowski decompositions of $P$, as we will recall below. On the other hand, by Theorem \ref{thmm}, we see that the deformations induced by the elements of $\bigoplus_{m \in \cT} T^1_{\paar}(-m)$ are connected to Laurent polynomials. In this section, we show how to combine these two types of deformations.

\subsection{The Cayley cone}\label{subsec 82}

Let us fix a Minkowski decomposition $P = P_1 + \cdots + P_m$ and let  
$\widetilde{\sigma}$ be the cone over the Cayley polytope  
$
P_1 * \cdots * P_m.
$  
More precisely, $\widetilde{\sigma}$ is generated by  
\begin{equation}\label{eq cayley}
    \{(P_1,e_1), (P_2,e_2), \dots, (P_m,e_m)\} \subset (N \oplus \ZZ^m)_{\RR},
\end{equation}
where $e_1, \dots, e_m$ is the standard basis of $\ZZ^m$ and  
$
(P_i, e_i) := \{(a, e_i) \mid a \in P_i\}.
$  

\newcommand{\cay}{S_{\operatorname{cay}}}

For $i \in \{1, \dots, m\}$ and $c \in M_\RR$, we choose a vertex $v(c)_i$ of $P_i$ where $\langle c, \cdot \rangle$ achieves its minimum.  
As we defined $\eta(c)$ for $c \in M$ in Definition \ref{d:eta(c)}, we now define  
$$
\eta_i(c) := -\min_{v\in P_i} \langle v, c \rangle = -\langle v(c)_i, c \rangle \in \ZZ.
$$  
The generators of  
$$
\cay := \widetilde{\sigma}^\vee \cap (M \oplus \ZZ^m)
$$  
are given by  
$$
(c_1, \eta_1(c_1), \dots, \eta_m(c_1)), \dots, (c_r, \eta_1(c_r), \dots, \eta_m(c_r)), (0, e_1), \dots, (0, e_m) \in M \oplus \ZZ^m,
$$  
where $c_1, \dots, c_r$ are as in \eqref{eg hilbbas}.  
For $\bfk \in \NN^r$, we define  
\begin{equation}\label{eq fkxz}
F_\bfk(\bfx, \bfz) := \bfx^\bfk - \bfx^{\partial(\bfk)} \prod_{i=1}^m z_i^{\eta_i(\bfk)} \in \CC[\bfx, \bfz] := \CC[x_1, \dots, x_r, z_1, \dots, z_m].
\end{equation}
Clearly,  
$$
\sum_{i=1}^m \eta_i(c) = \eta(c).
$$  

After substituting $z_i$ with $u + Z_i$ in $F_\bfk(\bfx, \bfz)$ (for $i = 1, \dots, m$) and denoting the resulting polynomial by $F_\bfk(\bfx, u, \bfZ)$,  
we obtain a flat map  
\begin{equation}\label{eq def cay 2}
    \tilde{\pi}: X_{P_1*\cdots *P_m}:=\spec \CC[\bfx, u, \bfZ] / (F_\bfk(\bfx, u, \bfZ) \mid \bfk \in \NN^r) \to \spec \CC[[\bfZ]],
\end{equation}
with fibre over $0$ equal to $X$, hence $\tilde{\pi}$ is a deformation of $X$. Indeed, we can immediately verify that \eqref{eq def cay 2} defines a flat map by introducing the relations  
    $$
    R_{\bfa, \bfk}(\bfx,u,\bfz) := F_{\bfa + \bfk}(\bfx,u,\bfz) - \bfx^\bfa F_\bfk(\bfx,u,\bfz) - \prod_{i=1}^m z_i^{\eta_i(\bfk)} F_{\bfa + \partial(\bfk)}(\bfx,u,\bfz),
    $$  
    and observing that replacing $z_i$ with $u + Z_i$ lifts $r_{\bfa, \bfk}$.

Let $\tilde{r}:=(r, r_1, \dots, r_m) \in M \oplus \mathbb{Z}^m$. 
If a Laurent polynomial $f$ is decomposable, say $f = f_1 \cdots f_m$, where $f_i \in \mathbb{C}[N]$, then we say that $f$ is \emph{$(r, r_1, \dots, r_m)$-mutable} if each $f_i$ is \emph{$(r, r_i)$-mutable} for all $i = 1, \dots, m$. In this case we have a one-parameter deformation of affine Gorenstein toric variety $X_{\newt(f_1)*\cdots*\newt(f_m)}$, given by a deformation pair $(\tr,Q)$, where $Q\subset (r=0)\cap N$ is a line segment of lattice length $1$. The corresponding deformation parameter we denote by $t_{\tr}$.

    \begin{proposition}\label{prop 0mutcomp}
     Let $f=f_1\cdots f_m$ be $0$-mutable Laurent polynomial. There exists a formal deformation $\{F_{\bfk}(\bfx,\bfz,\bft)\mid \bfk\in \NN^r\}$ of $X_{\newt(f_1)*\cdots*\newt(f_m)}$ over 
     $$
    \CC[[t_{(r,r_1,\dots,r_m)} \mid f\text{ is }(r,r_1,\dots,r_m)\text{-mutable}]].
    $$
         \end{proposition}
\begin{proof}
Let $\tilde{f}$ be the Laurent polynomial with $\newt(\tilde{f}) = \newt(f_1) * \cdots * \newt(f_m)$, where the coefficients on $\newt(f_i)$ are given by $f_i$. 

Because $f$ is $0$-mutable, we can mutate $\tilde{f}$ via mutations coming from $$(r, r_1, \dots, r_m) \in M \oplus \mathbb{Z}^m$$ to obtain $\tilde{g}$, where $\newt(\tilde{g})$ corresponds to some lattice point $$(n, n_1, \dots, n_m) \in N \oplus \mathbb{Z}^m.$$
Moreover, $X_{\newt(\tilde{g})}$ is unobstructed in
$$
\{ t_{\tilde{r}} \mid \tilde{r} = (r, r_1, \dots, r_m) \in M \oplus \mathbb{Z}^m, ~ \langle (r, r_i), (n, n_i) \rangle = 0 \text{ for every } i = 1, \dots, m \}.
$$
We conclude the proof using the same techniques as in the proofs of Theorems \ref{th main 1} and \ref{th main 2}.
\end{proof}

Note that if $f=f_1\cdots f_m$ is $(r,r_1,\dots,r_m)$-mutable, then $f$ is $(r,\sum_{i=1}^mr_i)$-mutable in the usual sense, i.e.\ in the sense of Definition \ref{def m mut}. Conversely, if $f$ is $(r,p)$-mutable with $(r,p)\in M\oplus \ZZ$, then there exists $p_1,\dots,p_m\in \ZZ$, with $\sum_{i=1}^mp_i=p$, such that $f$ is $(r,p_1,\dots,p_m)$-mutable and $f_i$ is $(r,p_i)$-mutable.

\begin{corollary}\label{cor sec 8}
     Let $f=f_1\cdots f_m$ be $0$-mutable Laurent polynomial. There exists a formal deformation of $X_{\newt(f)}$ over $\CC[[\bfZ,t_{(r,\sum_{i=1}^mr_i)}\mid \text{$f$ is $(r,r_1,\dots,r_m)$-mutable}]]$.
\end{corollary}
\begin{proof}
    Let $\{F_{\bfk}(\bfx,\bfz,\bft)\mid \bfk\in \NN^r\}$ be the formal deformation of $X_{\newt(f_1)*\cdots*\newt(f_m)}$ constructed in Proposition \ref{prop 0mutcomp}. After substituting $z_i$ with $Z_i+u$ and $t_{(r,r_1,\dots,r_m)}$ with $t_{(r,\sum_{i=1}^mr_i)}$, we easily see that we obtain a formal deformation of $X_{\newt(f)}$ over 
    $$\CC[[\bfZ,t_{(r,\sum_{i=1}^mr_i)}\mid \text{$f$ is $(r,r_1,\dots,r_m)$-mutable}]]=\CC[[\bfZ,t_{s}\mid \text{$s\in \cM(f)$}]].$$
\end{proof}

\begin{remark}
Note that the assumption that $f = f_1 \cdots f_m$ is $0$-mutable is crucial in the construction. This can be generalised to the case where one of the $f_i$, for $i = 1, \dots, m$, is arbitrary, and the others are $0$-mutable. We do not know how to prove that there exists a formal deformation of $X_{\newt(f_1)*\cdots* \newt(f_m)}$ over
$$
\mathbb{C}[[t_{(r, r_1, \dots, r_m)} \mid f \text{ is } (r, r_1, \dots, r_m)\text{-mutable}]],
$$
without assuming the above conditions on $f$, even though we believe the statement should hold.
\end{remark}

\begin{definition}
    We say that two maximally mutable irreducible Laurent polynomials $f$ and $g$ are \emph{deformation equivalent} if $\cM(f) = \cM(g)$. 
    More generally, two maximally mutable Laurent polynomials $f$ and $g$ are \emph{deformation equivalent} if they decompose as
    $$
    f = \prod_{i=1}^{n} f_i, \quad g = \prod_{i=1}^{n} g_i,
    $$
    where each $f_i$ and $g_i$ are irreducible, and $f_i$ is deformation equivalent to $g_i$ for all $i = 1, \dots, n$.
\end{definition}

\begin{remark}\label{conj rem}
If we have a deformation problem controlled by a differential graded Lie algebra  
$\mathfrak{g}$ with $\dim_\CC H^1(\mathfrak{g}),\dim_\CC H^2(\mathfrak{g})<\infty$, the solution of the Maurer-Cartan equation gives us a miniversal base space of the form $\spec \CC[[\bft]]/I$, where $\bft$ is a basis of $H^1(\mathfrak{g})$ (see \cite[Section 2]{She17} in an even more general case of $L_\infty$ algebra, \cite{Ste91} or \cite[Section 8]{Ste03}). The deformations of $X=X_P$ are controlled by the differential graded Lie algebra coming from the cotangent complex, which is quasi-isomorphic to the Harrison differential graded Lie algebra $\mathfrak{g}$ (see \cite{Lod92}, \cite{Fil18}). We have $H^1(\mathfrak{g})=T^1_X$, which is not finite dimensional and $H^2(\mathfrak{g})=T^2_X$, which is finite dimensional. 
We conjecture that in our case the solution of the Maurer-Cartan equation is also of the form 
\begin{equation}\label{eq pre form}
\spec \CC[[\bfTT]]/\cI,
\end{equation}
where $\bfTT$ is a basis of $T^1_X$ and $\cI$ is finitely generated involving only finitely many variables. 
\end{remark}

The results of this paper, together with those of \cite{Fil23}, suggest the following conjecture.

\begin{conjecture}\label{sing comp conj}
There exists a canonical bijective correspondence  
$
\kappa\colon\mathfrak{B}\to \mathfrak{A},
$
where $\mathfrak{A}$ is the set of components of the miniversal deformation space of the three-dimensional affine toric Gorenstein pair $\paar$, with $X=X_P$ and $\mathfrak{B}$ is the set of deformation equivalence classes of maximally mutable Laurent polynomials that have Newton polygon equal to $P$. Moreover, the deformation component is a smoothing component if and only if it corresponds to a $0$-mutable Laurent polynomial.
 \end{conjecture}

\subsection{Application to deformation of Fano toric varieties}\label{zad sec}

We conclude by discussing implications of our results for constructing Fano manifolds with very ample anticanonical bundles. 
Let $f$ be a Laurent polynomial such that $\newt(f)$ is a reflexive polytope. Consider the Gorenstein toric Fano variety $Y_{\newt(f)}$ associated with the spanning fan of $\newt(f)$. The affine Gorenstein toric variety $X_{\newt(f)}$ is the affine cone over $Y_{\newt(f)}$, and thus their deformation theories are connected by a comparison theorem (see, e.g., \cite{Kle79} or \cite[Section 2]{CI16}). The polytope $\newt(f)$ has only one interior lattice point, and we say that $f$ is \emph{projectively $(m,g)$-mutable} if it is $(m,g)$-mutable and the affine function $\varphi_m$ takes value $0$ at the interior lattice point of $P$. If $f$ is projectively $(m,g)$-mutable, then by the comparison theorem, we get a one-parameter deformation of the projective variety $Y_{\newt(f)}$, and we denote its parameter by $\bar{t}_{(m,g)}$.

If $f$ is $(m,g)$-mutable, we denote by $t_{(m,g)}$ the parameter corresponding to the one-parameter deformation of $X_{\newt(f)}$ determined by the deformation pair $(m, \newt(g))$.  
For any Laurent polynomial $f$, we define the sets of deformation parameters:
$$
\bft_f := \{t_{(m,g)} \mid f \text{ is } (m,g)\text{-mutable} \},
$$
$$
\bar{\bfT}_f := \{ \bar{t}_{(m,g)} \mid f \text{ is projectively } (m,g)\text{-mutable} \}.
$$

We hope that the results of this paper can be extended to show that any Laurent polynomial $f$ gives rise to a deformation of $X_{\newt(f)}$ over $\CC[[\bft_f]]$, and a deformation of $Y_{\newt(f)}$ over $\CC[[\bar{\bfT}_f]]$, whose general fibre is a Fano manifold with a very ample anticanonical bundle.

\end{document}